\documentclass[12pt]{amsart}
\usepackage{amsmath, amssymb, latexsym, amscd, amsthm,amsfonts,amstext}
\usepackage[mathscr]{eucal}
\theoremstyle{plain}
\newtheorem{lemma}{Lemma}[section]
\newtheorem{theorem}[lemma]{Theorem}
\newtheorem{corollary}[lemma]{Corollary}

\newtheorem{proposition}[lemma]{Proposition}
\newtheorem{definition}[lemma]{Definition}
\newtheorem{remark}[lemma]{Remark}
\newtheorem{example}[lemma]{Example}
\renewcommand{\div}{{\mathrm{div}}}
\renewcommand{\dim}{{\mathrm{dim}}}

\newcommand{\abs}[1]{\lvert #1\lvert}
\newcommand{\ddc}{{\mathrm{dd^{c}}}}
\newcommand{\supp}{{\mathrm{supp}}}

\newcommand{\rank}{{\mathrm{rank}}}
\begin{document}

\title[Holomorphic mappings into compact complex manifolds]{Holomorphic mappings into compact complex manifolds}

\thanks{The research of the authors is supported by an NAFOSTED grant of Vietnam (Grant No. 101.01-2011.29).}

\author{Do Duc Thai and Vu Duc Viet}

\address{Department of Mathematics \newline
Hanoi National University of Education\newline
136 XuanThuy str., Hanoi, Vietnam}
\email{ducthai.do@gmail.com}
\email{vuvietsp@gmail.com}


\maketitle       
\begin{abstract}   The purpose of this article is to show a second main theorem with  the explicit truncation level for holomorphic 
mappings of $ \mathbb{C} $ (or of a compact Riemann surface) into a compact complex manifold sharing divisors  in subgeneral position.
\end{abstract}
\tableofcontents
\section{Introduction and main results}

Let $\{H_j\}_{j=1}^q$ be hyperplanes of $\mathbb{C}P^{n}$.
Denote by $Q$  the index set $\{1,2, \cdots, q\}$.
Let $N\geq n$ and $q\geq N+1$.
We say that the family $\{H_j\}_{j=1}^q$ are in
$N$-\textit{subgeneral position} 
if for every subset $R\subset Q$ with the cardinality $|R|=N+1$
$$\bigcap_{j\in R} H_j =\emptyset.  $$
If they are in $n$-subgeneral position, we simply say that they are in {\it general position}.

Let $f:  \mathbb{C}^m \to \mathbb{C}P^{n}$ be a linearly nondegenerate meromorphic mapping and $\{H_j\}_{j=1}^q$ 
be hyperplanes in $N$-subgeneral position in $\mathbb{C}P^{n}$. 
Then the Cartan-Nochka's second main theorem
(see \cite{Noc83}, \cite{No05}) stated that 
$$
||\ \ (q-2N+n-1)T(r,f) \leq
\sum_{i=1}^q N^{[n]}(r, \div (f,H_i))+ o(T(r,f)).
$$

The above Cartan-Nochka's second main theorem plays an extremely important role in Nevanlinna theory, with many applications 
to Algebraic or Analytic geometry. Over the last few decades, there have been several results generalizing this theorem to abstract 
objects. The theory of the second main theorems for algebraically nondegenerate holomorphic curves into an arbitrary nonsingular 
complex projective variety $V$ sharing curvilinear divisors in general position in $V$ began about 40 years ago and  has grown into a huge 
theory. Many contributed. We refer readers to the articles \cite{CG}, \cite{GK}, \cite{No81}, \cite {No98}, \cite{NW02}, \cite{NWY02}, \cite{NWY08}, \cite{NW10}, \cite{R04}, \cite{R09}, \cite{Siu}, \cite{SY}  and references therein for the development of 
related subjects. 
We recall some  recent results and which are the best results available at present.

In  2004, Min Ru \cite{R04} established a defect
relation for algebraically nondegenerate holomorphic curves $f : \mathbb{C} \to \mathbb{C}P^{n}$ intersecting
curvilinear hypersurfaces in general position in $\mathbb{C}P^{n}$, which settled a long-standing conjecture
of B. Shiffman (see \cite{Shi}). Recently, in \cite{R09} he further extended the above mentioned
result to holomorphic curves $f : \mathbb{C} \to  V$ intersecting hypersurfaces  in general position in $V$, where $V$  is
an arbitrary nonsingular complex projective variety in $ \mathbb{C}P^{k}$. We now state his celebrated theorem.

Let $V \subset  \mathbb{C}P^{k}$  be a smooth complex projective variety of dimension $n\geq1.$ Let $D_1,\cdots,D_q$ be hypersurfaces in
$\mathbb{C}P^{k},$ where $q > n.$  Also, $D_1,\cdots,D_q$ are
said to be in general position in $V$  if for every subset $\{i_0,\cdots,i_n\}\subset \{1,\cdots,q\},$ 
$$V \cap suppD_{i_0} \cap\cdots\cap suppD_{i_n} =\emptyset,$$
where $suppD$ means the support of the divisor $D.$  A map $f : \mathbb{C}\to V$ is said to
be algebraically nondegenerate if the image of $f$ is not contained in any proper
subvarieties of $V.$

{\bf Theorem of Ru} (see \cite{R09})\ {\it  Let $V\subset \mathbb{C}P^{k}$ be a smooth complex
projective variety of dimension $n \geq 1.$ Let $D_1,\cdots,D_q$ be hypersurfaces in
$\mathbb{C}P^{k}$ of degree $d_j,$  located in general position in $V.$  Let $f : \mathbb{C} \to V$ be an
algebraically nondegenerate holomorphic map. Then, for every $\epsilon > 0,$
$$\sum_{j=1}^qd_j^{-1}m_f (r;D_j) \leq (n + 1 + \epsilon)T(r,f),$$
where the inequality holds for all $r \in (0, \infty)$ except for a possible set $E$ with
finite Lebesgue measure.}

As the first steps towards establishing the second main theorems for curvilinear divisors in subgeneral position in a (nonsingular) 
complex projective variety, recently, D. T. Do and V. T. Ninh in \cite{DN} and G. Dethloff, V. T. Tran and D. T. Do in \cite{DTT}
gave the Cartan-Nochka's second main  theorem with the truncation for  holomorphic curves $f : \mathbb{C} \to  V$ intersecting hypersurfaces located in $N$-subgeneral position in an arbitrary smooth complex projective variety $V.$ 
We now state their theorem in \cite{DTT}.

Let $N\geq n$ and $q\geq N+1.$
Hypersurfaces $D_1,\cdots, D_q$ in $\mathbb{C}P^{M}$ with
$V\not\subseteq D_j$ for all $j=1,...,q$
are said to be in $N$-subgeneral position in $V$ if the two following conditions are satisfied:

 (i)$\quad$ For every $1\leq j_0<\cdots<j_N\leq q,$ $V \cap D_{j_0}\cap\cdots\cap D_{j_N}=\varnothing.$

 (ii)$\quad$ For any subset $J\subset\{1,\cdots,q\}$ such that  $0<|J|\leq n$ and $\{D_j,\; j\in J\}$ are in general position in $V$
 and $V \cap (\cap_{j\in J}D_j)\not= \emptyset$, there exists an irreducible component $\sigma_J$ of $V \cap (\cap_{j\in J}D_j)$ with $\dim\sigma_J=\dim \big( V \cap (\cap_{j\in J}D_j)\big)$ such that for any $i\in\{1,\cdots,q\}\setminus J$, if $\dim \big( V \cap (\cap_{j\in J}D_j)\big)= \dim \big(V \cap D_i \cap (\cap_{j\in J}D_j)\big),$ then $D_i$ contains $\sigma_J.$ 

\vskip0.3cm

{\bf Theorem of Dethloff-Tran-Do} (see \cite{DTT})\ {\it  Let $V\subset \mathbb{C}P^{M}$ be a smooth complex projective variety of
dimension $n\geq 1.$ Let f be an algebraically nondegenerate
holomorphic mapping of $\mathbb{C} $ into $V.$ Let  $D_1,$ $\cdots,$ $D_q$
($V\not\subseteq D_j$) be hypersurfaces in $\mathbb{C}P^{M}$ of
degree $d_j,$ in $N$-subgeneral position in $V,$ where $N\geq n$ and $q\geq 2N-n+1.$ Then, for every $\epsilon>0,$ there exist  positive
integers $L_j\:(j=1,...,q)$ depending on $n, \deg V, N, d_j \:(j=1,...,q),q, \epsilon$ in an explicit way such that
\begin{align*}
\Big\Vert(q-2N+n-1-\epsilon)T_f(r)\leq \sum_{j=1}^q\frac{1}{d_j}N_f^{[L_j]}(r,D_j).
\end{align*}}

We would like to emphasize the following.\\
(i) The condition $(ii)$ in the above definition on $N$-subgeneral position of Dethloff-Tan-Thai is hard. Thus, their results may not be
very useful and applicable due to this reason.\\
(ii) In the above-mentioned papers and in other papers (see \cite{CRY}, \cite{ES} for instance), either there is no the truncation levels or the truncation levels obtained  depend on the given $\epsilon$.
When $\epsilon$ goes to zero, the truncation level goes to infinite (so the truncation
is totally lost). The most serious and difficult problem (which is supposed
to be extremely hard) is to get the truncation which is independent of $\epsilon$.\\
(iii) The family $D_1,\cdots,D_q$ are hypersurfaces in $\mathbb{C}P^{k},$ i.e they are global hypersurfaces. The more difficult problem is to consider the case where the family $D_1,\cdots,D_q$ are hypersurfaces in $V,$ i.e they are local hypersurfaces. 

\vskip0.2cm
Motivated by studing holomorphic mappings into compact complex manifolds, the following arised naturally.

{\bf Problem 1.}\ {\it To show a second main theorem with the explicit truncation level for  holomorphic mappings of $\mathbb{C}$ into a compact complex manifold sharing divisors in subgeneral position.}

\vskip0.2cm
Unfortunately, this problem is extremely difficult and while a substantial amount of
information has   been   amassed   concerning 
the second main theorem for holomorphic curves into complex projective varieties through the
years, the present knowledge of this problem for arbitrary compact complex manifolds has
remained extremely meagre. So far there has been no literature of such results.

The purpose of this paper is to solve  Problem 1 in the case where divisors are defined by global sections of a holomorphic line bundle over a given compact complex manifold. To state the results, we recall some definitions of Nevanlinna theory.

    Let $L \xrightarrow{\pi} X$ be a holomorphic line bundle over a compact complex manifold $X$. Let $R_1,\cdots,R_q$ be divisors of global sections of $H^{0}(X, L^{d_i})$ respectively, where $d_1,\cdots, l_q$ are positive numbers. Take a positive integer $d$ such that $d$ is divided by $lcm(d_1,d_2,\cdots,d_q)$. 
Let $E$ be  a  $\mathbb{C}$-vector subspace of dimension $m+1$  of $H^{0}(X, L^{d})$ such that $\sigma_{1}^{\frac{d}{d_1}},\cdots,\sigma_{q}^{\frac{d}{d_q}}\in E$. Take  a basis $\{c_k\}_{k=1}^{m+1}$ a basis of  $E$.  Then $\cap_{1\le i\le m+1}\{c_i=0\}=B(E)$ and $\sigma^{\frac{d}{d_i}}_i$ is a linear combination of $\{c_k\}_{k=1}^{m+1}$.  
Take  trivializations $p_i=(p^1_i,p^2_i): \pi^{-1}(U_i) \rightarrow U_i \times \mathbb{C}$  of $L$ and $\{U_i \}$ covers $X$.
 Denote by $\lambda_{ij}$ the transition function system of this local trivialization covering. 
Set $h_i =(\sum_{1\le i\le m+1}\abs{p^2_i \circ c_i}^2)^{1/d}$ on each $U_i$. It is easy to check that $h_i= \abs{\lambda_{ij}}^2 h_j$ on $U_i \cap U_j$. 
Put $h=\{h_i\}$. Then $\ddc \log h_j= \ddc \log h_i$ for $U_i \cap U_j \not = \varnothing$ and hence, $\ddc\log  h:= \ddc\log h_i$ is well-defined in $X-B(E)$.
 
 Let $f$ be a holomorphic mapping of $\mathbb{C}$ into $X$ such that $f(\mathbb{C})\cap B(E)=\varnothing$. We define the characteristic function of $f$ with respect to $E$ to be
             $$T_f(r,L)=\int_{1}^r \frac{1}{t}\int_{\abs{z}\le t} f^* \ddc \log h,$$ 
where $r>1$. Remark that the definition does not depend on choosing basic of $E.$ Take a section $\sigma$ of $L$. Assume that $D=\nu_{\sigma}=\sum_i s_i a_i$ is its zero divisor,
 where $a_i \in \mathbb{C}$ and $s_i$ is a positive integer. For $1\le k\le \infty,$ we put  $\nu^{[k]}(D)=\sum_{i}\min\{s_i,k\} a_i$ and   $n^{[k]}(t,D)=\sum_{\abs{a_i} <t}\min\{s_i,k\}$. The truncated counting function of $f$ to level $k$  with respect to $D$ is
            $$N^{[k]}_f (r,D)=\int_{1}^r \frac{n^{[k]}(t,D)}{t} dt.$$
For brevity we will omit the character $^{[k]}$ if $k=\infty$. Put $\lvert \lvert \sigma(\cdot)\lvert \lvert:= \dfrac{\sigma(\cdot)}{\sqrt{h(\cdot)}}$
By the Jensen formula and by the Poincare-Lelong theorem, we have 
\begin{theorem}(First main theorem) Let the notations be as above. Then,
               $$T_f (r,L)=\int_{0}^{2\pi} \log\frac{1}{\lvert \lvert \sigma(f(re^{i\phi}))\lvert \lvert} d\phi + N_f(r,D)+O(1).$$
\end{theorem}

Let $L \rightarrow X$ be a holomorphic line bundle over a compact complex manifold $X$ and $E$ be a $\mathbb{C}-$vector subspace  of $H^{0}(X,L).$   Let $\{c_k\}_{k=1}^{m+1}$ be a basis of $E$ and $B(E)$ be the base locus of $E.$ 
Define a mapping $\Phi: X - B(E) \rightarrow \mathbb{C}P^{m}$ by $\Phi (x):= [c_1(x):\cdots :c_{m+1}(x)]$. Denote by  $\rank E$ the maximal rank of Jacobian of $\Phi$ on $X-B(E)$. It is easy to see that this definition does not depend on choosing a basis of $E.$ 

We now state our results.    

\vskip0.2cm
\textbf{Theorem A.}\emph{ Let $X$ be a compact complex manifold. Let $L\rightarrow X$ be a holomorphic line bundle over $X$. Fix a positive integer $d$. Let $E$ be  a  $\mathbb{C}$-vector subspace of dimension $m+1$  of $H^{0}(X, L^{d})$. Put  $u=\rank E$ and $b=\dim B(E)+1$ if $B(E)\not =\varnothing$, otherwise $b=-1$. Take positive divisors $d_1, d_2,\cdots, d_q$  of $d$. Let $\sigma_j (j=1,2,\cdots,q)$ be in $H^{0}(X, L^{d_j})$ such that $\sigma_{1}^{\frac{d}{d_1}},\cdots,\sigma_{q}^{\frac{d}{d_q}}\in E$. Set $D_j=\{\sigma_j=0\}$ and denote by $R_j$ the zero divisors of $\sigma_j.$  Assume that $D_1,\cdots,D_q$ are in $N$-subgeneral position with respect to $E$ and $u>b.$  Let $f: \mathbb{C} \rightarrow X$ be an analytically non-degenerate holomorphic mapping with respect to $E$, i.e  $f(\mathbb{C})\not\subset \supp(\nu_{\sigma})$ for any $\sigma \in E\setminus \{0\}$ and $\overline{f(\mathbb{C})}\cap B(E)=\varnothing.$  
Then, 
   $$\Big\Vert (q-(m+1)K(E,N, \{D_j\}))T_f(r,L)\le \sum_{i=1}^{q} \frac{1}{d_i}N^{[m]}_f(r,R_i),$$
where 
$k_N,s_N,t_N \text{ are defined as in Proposition \ref{apro8}}\ and $
$$K(E,N, \{D_j\})=\dfrac{k_N(s_N-u+2+b)}{t_N}.$$}

We would like to point out that, in general,  the above constants $m, u$ do not depend on the dimension of $X$ (cf. Example \ref{aex1} below).

Since $u \le n_0(\{D_j\})\le n(\{D_j\}) \le m$ (cf. Subsection 2.1 below), we have a nice corollary in the case $m=u, B(E)=\varnothing.$

\textbf{Corollary 1.} \emph{Let $X$ be a compact complex manifold. Let $L\rightarrow X$ be a holomorphic line bundle over $X$. Fix a positive integer $d$. Let $E$ be  a  $\mathbb{C}$-vector subspace of dimension $m+1$  of $H^{0}(X, L^{d})$. Put  $u=\rank E.$ Take positive divisors $d_1, d_2,\cdots, d_q$  of $d$. Let $\sigma_j (j=1,2,\cdots,q)$ be in $H^{0}(X, L^{d_j})$ such that $\sigma_{1}^{\frac{d}{d_1}},\cdots,\sigma_{q}^{\frac{d}{d_q}}\in E$. Set $D_j=\{\sigma_j=0\}$ and denote by $R_j$ the zero divisors of $\sigma_j.$  Assume that $B(E)=\varnothing, \ m=u$ and $D_1,\cdots,D_q$ are in $N$-subgeneral position with respect to $E.$   Let $f: \mathbb{C} \rightarrow X$ be an analytically non-degenerate holomorphic mapping with respect to $E$, i.e  $f(\mathbb{C})\not\subset \supp(\nu_{\sigma})$ for any $\sigma \in E\setminus \{0\}.$  
Then, 
   $$\Big\Vert (q-2N+u-1)T_f(r,L)\le \sum_{i=1}^{q} \frac{1}{d_i}N^{[m]}_f(r,R_i).$$}

In the special case where $X= \mathbb{C}P^{n}, d=d_1=\cdots=d_q=1, L$ is the hyperplane line bundle over $\mathbb{C}P^{n}$ and $E=H^{0}(X, L),$ then $m=u=n.$ Thus,
 the original Cartan-Nochka's second main theorem is deduced immediately from Corollary 1.

\vskip0.3cm
As we know well, the Cartan's original second main theorem has been extended to holomorphic mappings $f $ from a compact Riemann surface into $\mathbb{C}P^{n}$ 
sharing hyperplanes located in general position in $\mathbb{C}P^{n}.$ For instance,  J. Noguchi \cite{NW10} has established the Nevanlinna theory for holomorphic mappings of 
compact Riemann surfaces into complex projective spaces, and obtained the second main theorem for hyperplanes with truncation level. 
Recently, Yan Xu and Min Ru in \cite{YM} also have obtained a similar second main theorem. Motivated by the above Problem 1 and observations, we now study the following problem.

{\bf Problem 2.}\ {\it To show the second main theorem with the explicit truncation level which is independent of $\epsilon$ for  holomorphic mappings of a compact Riemann surface into a compact complex manifold sharing divisors in subgeneral position.}

The next part of this article is to show the second main theorems with an explicit truncation level which is independent of $\epsilon$ for holomorphic curves of a compact Riemann surface into a compact complex manifold sharing divisors in $N$-subgeneral position.
To state the results, we recall some definitions of Nevanlinna theory.

 Let $f$ be a holomorphic mapping of a compact Riemann surface $S$ into a compact complex manifold $X$ of dimension $n$. Let $L\rightarrow X$ be a holomorphic line bundle and $\omega$ be a curvature form of a hermitian metric in $L$.
 Let $E$ and $h$ be as above. We can see that $f^* \log h$ is a singular metric in the pulled-back line bundle $f^* L$ of $L$ (see Section 3 for definitions). 
We put 
                 $$ T(f,L)=\int_{S} f^{*}\omega $$
Note that this definition does not depend on choosing $\omega$ and by $(iii)$ of Lemma \ref{lem31} we get $T(f,L)=\int_S f^* \ddc \log h$. 
Let $R$ be a divisor in $H^{0}(X,L)$. Put $f^{*}R= \sum a_i V_i$ (this is a finite sum). Then, the truncated counting function to level $T$ of $f$ with respect to $R$ is defined by
                 $$ N^{[T]}(f,R)=\sum_{i} \min\{T, a_i\} .$$

 We now state our results.    

\vskip0.3cm
\textbf{Theorem B.}\emph{ Let $S$ be a compact Riemann surface with genus $g$ and $X$ be a compact complex manifold of dimension $n$. Let $L\rightarrow X$ be a holomorphic line bundle over $X$. Fix a positive integer $d$. Let $E$ be  a  $\mathbb{C}$-vector subspace of dimension $m+1$  of $H^{0}(X, L^{d})$. Put  $u=\rank E,$ and $b= \dim B(E)+1$ if $B(E)\not =\varnothing,$ $b=-1$ otherwise.  Take positive divisors $d_1, d_2,\cdots, d_q$  of $d$. Let $\sigma_j (j=1,2,\cdots,q)$ be in $H^{0}(X, L^{d_j})$ such that $\sigma_{1}^{\frac{d}{d_1}},\cdots,\sigma_{q}^{\frac{d}{d_q}}\in E$. Set $D_j=\{\sigma_j=0\}$ and denote by $R_j$ the zero divisors of $\sigma_j$.  Assume that $R_1,\cdots,R_q$ are in $N$-subgeneral position in $X$ and $u>b.$ 
Let $f: S\rightarrow X$ be a holomorphic mapping such that $f$ is analytically nondegenerate with respect to $E$, i.e  $f(S)\not\subset \supp(\nu_{\sigma})$ for any $\sigma \in E\setminus \{0\}$ and $f(S)\cap B(E)=\varnothing.$    
Then, 
      $$(q-(m+1)K(E,N, \{D_j\}))T_f(r,L)\le \sum_{i=1}^{q} \frac{1}{d_i}N^{[m]}_f(r,R_i)+ A(d,L),
$$
where $k_N,s_N,t_N \text{ are defined as in Proposition \ref{apro8}}$ and
$$ \ K(E,N, \{D_j\})=\dfrac{k_N(s_N-u+2+b)}{t_N},$$
$$A(d,L)=\begin{cases}
 \dfrac{m(m+1)k_N(g-1)}{t_N} & \text{ if } g\ge 1\\
  0 & \text{ if } g=0.
\end{cases} $$}

\vskip0.2cm     
We would like to emphasize the following at this moment.
By Remark \ref{re50} below, in Theorem A and in Theorem B, we have
$$\liminf_{r\to \infty} \dfrac{T_f(r,L)}{\log r} >0.$$

Finally, we will give some applications of the above main theorems. Namely, we show a unicity theorem for holomorphic curves of  a compact Riemann surface into a compact complex manifold sharing divisors in $N$-subgeneral position.
Moreover, we also generalize the Five-Point Theorem of Lappan to a normal family from an arbitrary hyperbolic complex manifold to a compact complex manifold.

\vskip0.2cm
The paper is organized as follows. In Section 2, we give a definition of hypersurfaces located in $N$-subgeneral position and construct Nocka weights for divisors defined by sections of a holomorphic line bundle. 
 In Section 3, we introduce to Nevanlinna theory for holomorphic mappings from a compact Riemann surface into a compact complex manifold. In Section 4, we show some lemmas which needed later. In Section 5 and Section 6, we end the proof of our main theorems. In Section 7, some applications of the above main theorems are given.  

\section{Hypersurfaces in $N$-subgeneral position}
 
Let $X$ be a compact complex manifold of dimension $n$. Let $L\rightarrow X$ be a holomorphic line bundle over $X$. Take $\sigma_j \in H^{0}(X, L), D_j=\{\sigma_j=0\}$ and $R_j$ is the zero divisor of $\sigma_j \ (j=1,2,\cdots,q)$. Let $E$ be  a  $\mathbb{C}$-vector subspace of dimension $m+1$  of $H^{0}(X, L)$ containing $\sigma_1,\ldots,\sigma_q$.

\begin{definition}\label{def21}\ The hypersurfaces $D_1, D_2, \cdots, D_q$ is said to be located in  $N$-subgeneral position with respect to $E$ if  for any $1\le i_0< \cdots< i_N\le q $, we have $\cap_{j=0}^{N}D_{i_j}=B(E)$.
\end{definition}
    
Assume that $\{D_j\}$ is located in $N-$subgeneral position with respect to $E$. Let $\{c_k\}_{k=1}^{m+1}$ be a basis of $E.$ 
Put  $u=\rank E,$ $b=\dim B(E)+1$ if $B(E) \not =\varnothing$ and $b=-1$ if $B(E)=\varnothing.$ Assume that $u>b.$ 

   We set  $\sigma_i= \sum_{1\le j \le m+1}a_{ij} c_j,$ where $a_{ij}\in \mathbb{C}$. Define a mapping $\Phi: X \rightarrow \mathbb{C}P^{m}$ by $\Phi (x):= [c_1(x):\cdots :c_{m+1}(x)].$ 
The above mapping is a meromorphic mapping. Let $G(\Phi)$ be the graph of $\Phi$. Define 
 $$p_1: G(\Phi)\rightarrow X, p_2: G(\Phi) \rightarrow \mathbb{C}P^{m}$$
 by $p_1(x,z)=x, p_2(x,z)=z.$  Since $X$ is compact, $p_1,p_2$ is proper and hence, $Y=\Phi(X)=p_2(p_1^{-1}(X))$ is an algebraic variety of $\mathbb{C}P^{m}$. Moreover, by definition of $\rank E$, $Y$ is of dimension $\rank E=u.$ 
Denote by $H$ the hyperplane line bundle of $\mathbb{C}P^{m}$. Put $H_i:= \sum_{1\le j \le m+1}a_{ij} z_{j-1}$, where $[z_0,z_1,\cdots,z_m]$ is the homogeneous coordinate of $\mathbb{C}P^m$. 

     For each $K\subset Q$, put $c(K)=\rank \{H_i\}_{i\in K}.$
We also set  
 $$n_0(\{D_j\})=\max\{ c(K):  K\subset Q \text{ with } \abs{K}\le N+1\}-1,$$
and
 $$n(\{D_j\})=\max\{ c(K):  K\subset Q\}-1.$$ 
Then $n(\{D_j\}),n_0(\{D_j\})$ are independent of the choice the $\mathbb{C}$-vector subspace $E$ of $H^{0}(X, L)$ containing  
$\sigma_j (1\le j\le q).$ By Lemma \ref{le:2.1} below, we see that 
               $$u \le n_0(\{D_j\})\le n(\{D_j\})\le m.$$

\begin{theorem} \label{athe1} Let notations be as above. Assume that 
  $R_1,\cdots,R_q$ are in $N$-subgeneral position with respect to $E$ and $q\ge 2N-u+2+b.$ Then, there exist Nocka weights $\omega(j)$ for $\{D_j\}$, i.e  they satisfy properties in Proposition \ref{apro8}. 
\end{theorem}

In order to prove Theorem \ref{athe1} we need the following lemmas.


\begin{lemma} \label{lem0}
Let notations be as above. Then $\dim p_2(p_1^{-1}(B(E)))\le \dim B(E)+1$. Moreover, $\dim p_2(p_1^{-1}(B(E)))=0$ if $B(E)=\varnothing.$
\end{lemma}
\begin{proof}
The second assertion is trivial. Let us consider the case of $B(E)\not= \varnothing.$ We have $\cap_{j=0}^{N} D_j= B(E) \not =\varnothing.$ Denote by $j_0$ the biggest index among $\{0,1,\cdots,N\}$ such that $\cap_{j=0}^{j_0} D_j \not= B(E).$ Using the fact that if the intersection of  a hypersurface and an analytic set of dimension $\alpha$ is not empty, its dimension is at least $\alpha -1,$ we see that  $\dim \cap_{j=0}^{j_0} D_j= \dim B(E)+1.$ Put $V=\cap_{j=0}^{j_0} D_j$. Then $\dim V=\dim B(E)+1$, $V$ is compact and $B(E) \subset V$. Consider the restriction $\Phi_1$ of $\Phi$ into $V.$ Then $\Phi_1$ is a meromorpic mapping between $V$ and $\mathbb{C}P^{m}$. Therefore, we get $\dim p_2(p_1^{-1}(B(E))) \le \dim \Phi_1(V)\le \dim V= \dim B(E)+1.$ 
\end{proof}
\begin{lemma}\label{le:2.1}
 Let $Y$, $H_j$ be as above. Then, $\rank\{H_j, 0\le j\le N\}\ge u-b$. 
 \end{lemma}
\begin{proof} Put $k=\rank\{H_j,j\in R\}$. Then, there are $0\le j_1<\cdots<j_k \le N$ such that $\rank \{H_{j_1},\cdots,H_{j_k}\}=k$ and $H_j$ ($j\in \{0,1,\cdots,N\}$) is a linear combination of $H_{j_1},\cdots, H_{j_k}$. Therefore, we have $H_0 \cap H_1\cdots \cap H_{N}\cap Y =H_{j_1}\cap \cdots \cap H_{j_k}\cap Y$.  We have 
\begin{align*}
p_1(p_2^{-1}(H_0 \cap H_1\cdots \cap H_{N}\cap Y )) &=p_1(p_2^{-1}(H_0) \cap \cdots \cap p_2^{-1}(H_N)\cap G(\Phi))\\
  &\subset D_0 \cap \cdots \cap D_N=B(E)
\end{align*}
Hence, $p_2^{-1}(H_0 \cap H_1\cdots \cap H_{N}\cap Y )\subset p_1^{-1}(B(E))$. That means 
      $$ H_0 \cap H_1\cdots \cap H_{N}\cap Y \subset p_2(p_1^{-1}(B(E))).$$
Therefore, by  Lemma \ref{lem0}, we get $\dim H_0 \cap H_1\cdots \cap H_{N}\cap Y \le b.$      
On the other hand, since $H_{j_1}\cap\cdots \cap H_{j_k}$ is an algebraic variety of dimension $m-k,$ it implies that  $b\ge m-k+u-m.$ This yields that $k\ge u-b$.           
\end{proof}

\begin{lemma} \label{lem2} For each $K\subset Q$, $c(K)\le \abs{K}$. And for $K\subset K^{'} \subset Q$ with $c(K^{'})=\abs{K^{'}},$ we have  $c(K)=\abs{K}$. 
\end{lemma}
\begin{proof} 
 The proof is trivial.
\end{proof}
\begin{lemma} \label{lem3} Let $K,R \subset Q$ such that $K\subset R$ and $c(K)=\abs{K}$. Then, there exists a set $K^{'}$ such that $K\subset K'\subset R$ and $c(K')=\abs{K'}= c(R).$ 
\end{lemma}
\begin{proof}   The proof is trivial.     
\end{proof}
\begin{lemma} \label{lem4} i) Let $R_1, R_2 \subset Q$. Then, 
     $$c(R_1 \cup R_2)+ c(R_1 \cap R_2)\le c(R_1)+ c(R_2)$$ 
ii) Let  $S_1 \subset S_2\subset Q$. Then, $\abs{S_1}- c(S_1)\le \abs{S_2}- c(S_2)$. Furthermore, if $\abs{S_2}\le N+1,$ then $\abs{S_2}-c(S_2)\le N-u+b+1$.
\end{lemma}
\begin{proof} i) By Lemma \ref{lem2}, there exist subsets $K,K_1,K_2$ with $K\subset R_1 \cap R_2$, $K\subset K_1\subset R_1$, $K_1\subset K_3\subset R_1\cup R_2$ such that
        $$\abs{K}=c(K)= c(R_1 \cap R_2), \abs{K_1}=c(K_1)=c(R_1),$$ 
and        
         $$\abs{K_3}=c(K_3)=c(R_1\cup R_2).$$
Set $K_2=K_3-K_1$. We show that  $K_2\subset R_2$. Indeed, otherwise there exists $i\in K_2- R_2$. Then $i \in R_1-K_1$ and hence, $K_1\cup \{i\}\subset K_3$ and $K_1\cup \{i\}\subset R_1$. This implies that if  $\abs{K_1}=c(K_1)=c(R_1),$ then 
$c(R_1)\ge c(K_1\cup \{i\})= \abs{K_1\cup \{i\}}= c(K_1)+1=c(R_1)+1.$ This is a contradiction. Thus, $K_2\subset R_2$ and hence, $K_2\cup K \subset R_2$. On the other hand, $K_2\cup K\subset K_3$ and $K_2\cap K\subset K_2\cap K_1=(K_3-K_1)\cap K_1=\varnothing$. By Lemma \ref{lem2}, we get $c(R_2)\ge c(K_2\cup K)=\abs{K_2\cup K}=\abs{K_2}+\abs{K}=(\abs{K_3}- \abs{K_1})+ \abs{K}\ge c(R_1\cup R_2)- c(R_1)+ c(R_1 \cap R_2)$. Hence, the assertion $(i)$ holds.\\
ii) By Lemma \ref{lem2}, there exist $S^{'}_{v}\, (v=1,2)$ such that $S^{'}_v \subset S_v$, $S^{'}_1 \subset S^{'}_2$ and $\abs{S^{'}_v}= c(S^{'}_v)=c(S_v)$. We show that  $(S^{'}_2 -S^{'}_1)\cap S_1= \varnothing$. Indeed, otherwise there exists $i\not \in S^{'}_1$ such that $S^{'}_1 \cup\{i\}\subset S^{'}_2$ and $S^{'}_1 \cup\{i\}\subset S_1$.  If $\abs{S^{'}_1}= c(S^{'}_1)=c(S_1),$ then $\abs{S^{'}_1}=c(S_1)\ge c(S^{'}_1 \cup \{i\})= c(S'_1)+1.$   This is  a contradiction. Thus, $(S^{'}_2 -S^{'}_1)\cap S_1= \varnothing$ and hence, $S^{'}_2- S^{'}_1 \subset S_2 - S_1$. Therefore, $c(S_2)-c(S_1)\le \abs{S^{'}_2}-\abs{S^{'}_1}+1=\abs{S^{'}_2-S^{'}_1}\le \abs{S_2-S_1}=\abs{S_2}-\abs{S_1}$. 

  If $\abs{S_2}\le N+1$, then we choose $S_3$ such that $S_2 \subset S_3\subset Q$ and $\abs{S_3}=N+1$. By Lemma \ref{le:2.1} , we have $c(S_3)\ge u-b$. Hence, $c(S_2)-\abs{S_2}\le N-u+b+1$. 
\end{proof}
For $R_1 \subsetneq R_2 \subset Q$, we set $\rho(R_1,R_2)=\dfrac{c(R_2)-c(R_1)}{\abs{R_2}-\abs{R_1}}$.
\begin{lemma} \label{lem5}   Let the notations be as above and assume that $q\ge 2N-u+2+b.$ Then, there exists a sequence of subsets $\varnothing:=R_0\subsetneq R_1\subsetneq \cdots \subsetneq R_s\subset Q$ satisfying the following conditions:\\
i) $c(R_s)<u-b$,\\
ii) $0< \rho(R_0,R_1)< \cdots <\rho(R_{s-1},R_s)\le \dfrac{u-b-c(R_s)}{2N-u+2+b-\abs{R_s}}$,\\
iii) For any $R$ with $R_{i-1}\subsetneq R\subset Q$ ($1\le i\le s$) and $c(R_{i-1})< c(R)< u-b$,  we get $\rho(R_{i-1},R_i)\le \rho(R_{i-1},R)$. Moreover, if $\rho(R_{i-1},R_i)= \rho(R_{i-1},R),$ then 
$\abs{R}\le \abs{R_s}$.\\
iv) For any $R$ with $R_s \subsetneq R\subset Q$, if $c(R_s)< c(R)< u-b,$  then $\rho(R_s,R)\ge \dfrac{u+1-c(R_s)}{2N-u+2+b-\abs{R_s}}.$ 
\end{lemma}
\begin{proof}
The proof is similar to that of Lemma 2.4 in \cite{DTT}.
\end{proof}
\begin{lemma} \label{lem6} Let the notations be as above and assume that $q\ge 2N-u+2+b.$ Then, there exist constants $\omega'(j)\ (j\in Q)$ and $\Theta'$ satisfying the following conditions:\\
i) $0< \omega'(j) \le \Theta'$ ($j\in Q$), $\Theta'\ge \frac{u-b}{2N-u+2+b}.$\\
ii) $\sum_{j\in Q} \omega'(j)\ge\Theta'(\abs{Q}-2N+u-b-2)+ u-b$.\\
iii) If $ R\subset Q $ and $0\le \abs{R}\le N+1,$ then $\sum_{j\in R}\omega'(j) \le (n(\{D_j\})-u+2+b)c(R)$.
\end{lemma}
\begin{proof} 
By the condition $(i)$ of Theorem \ref{athe1}, we have $\abs{R_s}\le N$. 

Take a subset $R_{s+1}$ of $Q$ such that $\abs{R_{s+1}}=2N-u+2+b\ge N+1$ and $R_s\subset R_{s+1}$. 
Set 
    $$\Theta'= \rho(R_s,R_{s+1})=\frac{c(R_{s+1})-c(R_s)}{2N-u+2+b-\abs{R_s}},$$
and 
$$\omega'(j)=\begin{cases}
              \rho(R_i,R_{i+1}) &\text{ if $j\in R_{i+1}-R_i$ for some $i$ with $1\le i \le s$},\\
               \Theta' & \text{ if $j \not \in R_s$}.     
              \end{cases}$$
By $c(R_{s+1})\le n(\{D_j\})$, we have
                  $$\frac{u-b-c(R_s)}{2N-u+2+b-\abs{R_s}}\le \Theta'' \le  \frac{n(\{D_j\})-c(R_s)}{2N-u+2+b-\abs{R_s}}.$$                            
By Lemma \ref{lem5} ii), we get 
\begin{align}\label{eq:04}
\omega'(j)\le \Theta'\text { for all }j\in Q.
\end{align}
 We have
\begin{align*}
\sum_{j=1}^{q}\omega'(j)&= \sum_{j\in Q-R_{s+1}}\omega'(j) +\sum_{i=0}^{s}\sum_{j\in R_{i+1}-R_i} \omega'(j)\\
              &= \Theta'(q- 2N+u-b-2)+ \sum_{i=0}^{s}(c(R_{i+1})- c(R_i))\\
              & \ge \Theta'(q- 2N+u-b-2)+ u-b. 
\end{align*} 
This yields that
              $$\Theta'(q-2N+u-2-b)+u-b\le q\Theta'.$$
Hence $\Theta'\ge \frac{u-b}{2N-u+2+b}.$
Combining with (\ref{eq:04}), we see that $\omega'(j)$ and $\Theta'$ satisfy $(i)$ and $(ii).$

We now check the condition $(iii).$ Take an arbitrary subset $R$ of $Q$ with $0< \abs{R}\le N+1$.\\

 \textbf{Case 1.} $c(R\cup R_s)\le u-b-1$.
 
  Set 
  $$R^{'}_{i}=\begin{cases}
             R\cap R_i & \text{ if } 0\le i\le s,\\
             R  & \text{ if } i=s+1 
             \end{cases} $$
We show that for any $i\in \{1, \cdots,s+1\}$, if $\abs{R^{'}_{i}}> \abs{R^{'}_{i-1}}$ then
\begin{align}\label{eq:5}                   
c(R^{'}_{i}\cup R_i)> c(R_{i-1}).
\end{align}
and 
\begin{align}\label{eq:6}                     
\rho(R_{i-1}, R_i)\le \rho(R^{'}_{i-1}, R^{'}_{i}).
\end{align}
* If $i=1:$ Since $\abs{R^{'}_1}>\abs{R^{'}_0}=0,\ R^{'}_1\not =\varnothing.$ Then $c(R^{'}_{1}\cup R_0)= c(R^{'}_1)> 0=c(R_{0}).$ \\
* If $i\ge 2:$ Since $\abs{R^{'}_i}> \abs{R^{'}_{i-1}},$  it implies that $\abs{R^{'}_i\cup R_{i-1}}> \abs{R_{i-1}}$. On the other hand, we have $c(R_{i-2})<c(R_{i-1})\le c(R^{'}_i\cup R_{i-1})\le c(R\cup R_s)\le n$. By Lemma \ref{lem5} $(iii),$ we get $\rho( R_{i-2}, R_{i-1})< \rho(R_{i-2}, R^{'}_i \cup R_{i-1})$. This yields that 
   $$\frac{c(R_{i-1}-R_i)}{\abs{R_{i-1}}- \abs{R_{i-2}}}< \frac{c(R^{'}_i\cup R_{i-1})-c(R_{i-2})}{\abs{R^{'}_i\cup R_{i-1}}- \abs{R_{i-2}}}.$$
Since $\abs{R_{i-1}}< \abs{R^{'}_i\cup R_{i-1}},$ it implies that $c(R_{i-1})< c(R^{'}_i\cup R_{i-1})$. The inequality (\ref{eq:5}) is proved. 

 We now prove (\ref{eq:6}). 

By (\ref{eq:5}), $c(R_{i-1})< c(R^{'}_i\cup R_{i-1})\le c(R\cup R_s)\le n$. By Lemma \ref{lem5} $(iii)$ for the case $1\le i\le s$ and $(iv)$ for the case $i=s+1$, we have 
    $\rho(R_{i-1}, R_i)\le \rho(R^{'}_{i-1}, R^{'}_{i})\ (1\le i\le s+1),$ where
$\rho(R_s,R_{s+1})=\dfrac{n+1-c(R_s)}{2N-n+1-\abs{R_s}}.$ Therefore, by Lemma \ref{lem4}, we have
\begin{align*}
\rho(R_{i-1}, R_i)&\le \rho(R^{'}_{i-1}, R^{'}_{i})\\
                  &=\frac{c(R^{'}_i\cup R_{i-1})-c(R_{i-2})}{\abs{R^{'}_i\cup R_{i-1}}- \abs{R_{i-2}}}\le 
\frac{c(R^{'}_i) -c(R^{'}_i\cap R_{i-1})}{\abs{R^{'}_i\cup R_{i-1}}- \abs{R_{i-1}}}\\
                  &=\frac{c(R^{'}_i)-c(R^{'}_{i-1})}{\abs{R^{'}_i}- \abs{R^{'}_{i-1}}}\\
                  &=\rho(R^{'}_{i-1},R^{'}_i).  
\end{align*}
The inequality (\ref{eq:6}) is proved. By (\ref{eq:6}), we see that
\begin{align}\label{eq:7} 
\omega'(j)\le \rho(R^{'}_{i-1}, R^{'}_i) \text{ for all } j\in R^{'}_i - R^{'}_{i-1}\, (1\le i\le s+1).
\end{align}
By (\ref{eq:7}), we have 
\begin{align*}
\sum_{j\in R}\omega'(j) &= \sum_{i=1}^{s+1} \sum_{j\in R^{'}_i- R^{'}_{i-1}} \omega'(j)\\
                      &\le \sum_{i : R^{'}_i- R^{'}_{i-1}\not =\varnothing} (\abs{R^{'}_i}- \abs{R^{'}_{i-1}}). (\rho(R^{'}_{i-1},R^{'}_i) \\
                      &\le c(R^{'}_{s+1})- c(R^{'}_0)=c(R).
\end{align*}
Hence, the assertion $(iii)$ holds in this case.

 \textbf{Case 2.} $c(R\cup R_s)\ge u-b$. 

By Lemma \ref{lem4} and since $\abs{R}\le N+1$, we have 
        $$\abs{R}\le c(R)+N-u+b+1$$
and 
        $$u-b-c(R_s)= c(R\cup R_s)- c(R_s)\le c(R)- c(R\cap R_s)\le c(R).$$

By $\omega'(j)\le \Theta'$, we have
\begin{align*}
\sum_{j\in R} \omega'(j)\le \Theta'\abs{R}&\le \Theta'(c(R)+N-u+b+1)\\
                                               &=\Theta' \cdot c(R)\cdot \bigg(1+\frac{N-u+b+1}{c(R)}\bigg)\\
                                               &\le \Theta' \cdot c(R)\cdot \bigg(1+\frac{N-u+b+1}{u-b-c(R_s)}\bigg)\\
                                               &=\Theta' \cdot c(R)\cdot \frac{N+1-c(R_s)}{u-b-c(R_s)}\\
                                               &\le c(R)\cdot \frac{n(\{D_j\})-c(R_s)}{2N-u+2+b-\abs{R_s}}\cdot \frac{N+1-c(R_s)}{u-b-c(R_s)}\\
                                                &\le (n(\{D_j\})-u+2+b)c(R).
\end{align*}          
Lemma \ref{lem6} is proved.
\end{proof}

\begin{proposition}\label{pro7} Let the notations be as above and assume that $q\ge 2N-u+2+b$.  Take arbitrary non-negative real constants $E_1, E_2,\cdots, E_q$ and  a subset $R$ of $Q$ with $\abs{R}\le N+1.$ Then, there exist $j_1, j_2, \cdots, j_{c(R)} \in R$  such that 
$$\cap_{1\le i\le c(R)}H_{j_i} \cap Y=\cap_{j\in R}H_j\cap Y$$ 
and
$$ \sum_{j\in R}\omega'(j )E_j \le  (n(\{D_j\})-u+2+b)\sum_{ 1\le i\le c(R)}E_{j_i}.$$   
\end{proposition}
\begin{proof} 
Without loss of generality, we may assume that $E_1\ge E_2 \ge \cdots \ge E_q$. We will choose $j_i^{'}s$ by induction on $i$. Firstly, choose $j_1=1$ and set 
$K_1= \{l\in R: c(\{j_1, l\})= c(\{j_1\})=1 \}$. Next, choose 
                    $$ j_2= \min\{t: t\in R-K_1 \}$$
and set $K_2=\{l\in R: c(\{j_1, j_2, l\})= c(\{j_1, j_2\}) \}$. Similarly, choose
                    $$ j_3= \min\{t: t\in R-K_2 \}.$$
We continue this process to obtain $K_1 \subsetneq K_2 \subsetneq \cdots \subsetneq K_{t} =R$. 
It is clear that $c(K_i)= i$ for 
each $1\le i\le t$. By the choice of $j_i (1\le i\le t),$ we have $c(\{j_1, j_2,\cdots,j_t\})=c(R)$ and hence, $\cap_{1\le i\le c(R)}H_{j_i}=\cap_{j\in R}H_j.$ 

  Set $K_0:= \varnothing$ and $a_i:= \sum_{j\in K_i- K_{i-1}} \omega'(j)\ (1\le i\le t)$. Then, by  Lemma \ref{lem6}, we have 
 \begin{align*}
\sum_{j=1}^{i}a_j = \sum_{j\in K_i} \omega'(j) &\le (n(\{D_j\})-u+2+b)c(K_i)\\
  & \le (n(\{D_j\})-u+2+b)i\, (1\le i\le t)
\end{align*}
On the other hand, it is easy to see that  $E_j\le E_{j_i}$ for each  $1\le i\le t$ and $j\in K_i - K_{i-1}$. Thus, we get
\begin{align*}
\sum_{j\in R}\omega'(j) E_j & =\sum_{i=1}^{t} \sum_{j\in K_i-K_{i-1}} \omega'(j)E_j\\
                  & \le \sum_{i=1}^{t} \sum_{j\in K_i-K_{i-1}} \omega'(j)E_{j_i}=\sum_{i=1}^{t}a_i E_{j_i}\\
                  &= \sum_{i=1}^{t-1}(a_1 +\cdots+ a_i)(E_{j_i}- E_{j_{i+1}})+ (a_1+\cdots+ a_{j_t}) E_{j_{t}}\\
                  & \le  (n(\{D_j\})-u+2+b)(\sum_{i=1}^{t-1} i(E_{j_i}-E_{j_{i+1}})+ t E_{j_{n}})\\
                  &= (n(\{D_j\})-u+2+b)(E_{j_1}+E_{j_2}+\cdots+ E_{j_{t}})\\
                  &\le (n(\{D_j\})-u+2+b)(E_{j_1}+E_{j_2}+\cdots+ E_{j_{t}}).
\end{align*}
Hence, 
$$\sum_{j\in R}\omega'(j) E_j \le (n(\{D_j\})-u+2+b)  (E_{j_1}+E_{j_2}+\cdots+ E_{j_{t}}). $$
The proof is completed.
\end{proof}
\begin{proposition}\label{apro7} Let the notations be as above and assume that $q\ge 2N-u+2+b$.  Take arbitrary non-negative real constants $E_1, E_2,\cdots, E_q$ and  a subset $R$ of $Q$ with $\abs{R}=N+1.$ Then, there exist $j_1, \cdots, j_{n_0(\{D_j\})+1}$ in $R$  such that 
$$\cap_{1\le i\le n_0(\{D_j\})+1}H_{j_i} \cap Y=\cap_{j\in R}H_j \cap Y$$ 
and
$$ \sum_{j\in R}\omega'(j )E_j \le  (n(\{D_j\})-u+2+b)\sum_{ 1\le i\le n_0(\{D_j\})+1}E_{j_i}.$$   
\end{proposition}
\begin{proof}
By the definition of $n_0(\{D_j\})$, we have $c(R)\le n_0(\{D_j\})+1$. By Proposition \ref{pro7}, there exist 
$j_1, j_2, \cdots, j_{c(R)} \in R$  such that 
$$\cap_{1\le i\le c(R)}H_{j_i}=\cap_{j\in R}H_j$$ 
and
$$ \sum_{j\in T}\omega'(j )E_j \le  (n(\{D_j\})-u+2+u)\sum_{ 1\le i\le c(R)}E_{j_i}.$$
Take $j_{c(R)+1}, \ldots,j_{n_0(\{D_j\})+1}\in R.$ Then $\{j_1, \ldots, j_{n_0(\{D_j\})+1}\}$ satisfies the conclusion. The proof 
is completed.                    
\end{proof}
 
\begin{proposition}\label{apro8} Put $k_N=2N-u+2+b, s_N=n_0(\{D_j\})$ and $t_N=\dfrac{u-b}{n(\{D_j\})-u+2+b}.$ Assume that $q\ge k_N$.  Then there exist constants $\omega(j)\ (j\in Q)$ and $\Theta$ satisfying the following conditions:
\begin{itemize}
\item[(i)] \  $0< \omega(j) \le \Theta \, (j\in Q), \Theta \ge t_N/k_N$.
\item[(ii)] \  $\sum_{j\in Q} \omega(j)\ge \Theta (q-k_N)+ t_N.$
\item[(iii)] \  Let $E_j$ ($j\in Q$) be arbitrary positive real numbers
and $R$ be  a subset  of $Q$ with $\abs{R}=N+1.$ Then, there exist $j_1, \cdots, j_{s_N+1}$ in $R$  such that 
$$\cap_{1\le i\le s_N+1}H_{j_i} \cap Y=\cap_{j\in R}H_j \cap Y$$ 
and
$$ \sum_{j\in R}\omega(j )E_j \le  \sum_{ 1\le i\le s_N+1}E_{j_i}.$$
\end{itemize}
\end{proposition}
\begin{proof} Put 
     $$\Theta=\frac{\Theta'}{n(\{D_j\})-u+2+b}, \omega(j)=\frac{\omega'(j)}{n(\{D_j\})-u+2+b} \text{ for } j\in Q.$$
 From the above lemmas, it is easy to see that $\omega(j), k_N, t_N,s_N$ satisfy the requirements. 
\end{proof}

\section{Basic notions and auxiliary results from Nevanlinna theory}

  Let $L\rightarrow S$ be a holomorphic line bundle over a compact Riemann surface $S$. Denote by $\Gamma(S,L)$ the set of meromorphic sections of the holomorphic line bundle $L$. Let $D=\sum_{a\in S}\lambda_{a} a$ be a divisor in $\Gamma(S,L)$ (This sum is finite). Set
       $$N(D)=\sum_{a\in S}\lambda_{a}. a ,$$ 
and 
        $$N^{[k]}(D)=\sum_{a\in S}\min\{ k, \lambda_a\} \text { for } k\in \mathbb{Z}_+.$$   
    
  Let $g$ be a holomorphic function in an open subset $U$ of $S$. For $a\in U,$  denote by $\nu_{g}(a)$ the multiplicity at $a$ of the equation $g(x)=0$ and by $(g)_0$ the zero divisor of $g$. Let $(U_i \cap U_j, \xi_{ij})$ be a transition function system of $L$ and $\omega_L$ be the curvature form of a hermitian metric $\{h_i\}$ of $L$. Let $\sigma \in \Gamma(S,L)$. Denote by $\sigma_i$ the restriction of $\sigma$ on $U_i$. 

We say that $\psi=\{\psi_i \}$ is a \emph{singular metric} in $L$ if $\psi_i\ge 0$ is a nonnegative function on $U_i$ such that the following are satisfied\\
i) $\psi_i = \abs{\xi_{ij}}^2 \psi_j $ for $U_i \cap U_j \not =\varnothing$.\\
ii) $\log \psi_i$ is locally integrable.\\
As a current on $S$, the curvature current of $\psi$ is defined by
          $$\omega_{\psi}=\ddc[\log \psi_i] $$
\begin{lemma}\label{lem31} Let the notations be as above. Then, the following equations holds.\\
i)   $\ddc[\log \dfrac{1}{||\sigma(x)||^{2}}]= \omega_L - D,$ where $D=(\sigma)_{0}$ and $||\sigma(x)||^{2}=\dfrac{\abs{\sigma_i(x)}^{2}}{h_i}$ for each $x\in U_i$.\\
ii)  $\int_{S}\omega_L= N(D)$.\\
iii) $\int_{S} \omega_{\psi}= \int_{S}\omega_L$. 
\end{lemma}
\begin{proof} 
i) Take a partition of unity $\{c_i\}$ subordinated to the covering $\{U_i\}$ of $S$ and $f$ is a differential function on $S$. Then,
\begin{align*}
\ddc[\log \dfrac{1}{||\sigma(x)||^{2}}](f)&=\ddc[\log \dfrac{1}{||\sigma(x)||^{2}}](\sum_{i}c_i f)
\end{align*}
\begin{align*}
                                          &=\sum_{i}\ddc[\log \dfrac{1}{||\sigma(x)||^{2}}](c_i f)\\
                                          &=\sum_{i}(\ddc[\log h_i](c_i f)- \ddc[\log \sigma_i](c_i f))\\                                          
                                          &=\sum_{i}(\omega_L (c_i f)- (\sigma_i)_0(c_i f))\\
                                          &(\text{by Poincare-Lelong formular})\\
                                          &=\omega_L(f)-D(f).                                              
\end{align*}
ii) The desired formula is obtained from $(i)$ by integrating over $S$.\\
iii) By the construction of $\omega_L, \omega_{\psi}$, there is a locally integrable function $\phi$ such that $\omega_{\psi}- \omega_L= \ddc[\log \phi]$. By integrating over $S$, we get $\int_{S} \omega_{\psi}= \int_{S}\omega_L$.
\end{proof}
Let $f$ be a holomorphic mapping of $S$ into a complex manifold $X$. Let $L^{'}$ be a holomorphic line bundle over $X$. The pulled-back holomorphic line bundle of $L^{'}$ on $S$ by $f$ is denoted by $f^{*}L^{'}$. By Lemma \ref{lem31},
               $$T(f,L^{'})= \int_{S} f^{*}\omega_{L^{'}} .$$       
does not depend on choosing curvature form $\omega_{L^{'}}$ of $L^{'}$. We call it \emph{the characteristic function} of $f$ with respect to $L$.\\

 Let $d$ be a positive integer and $\Gamma(X,(L^{'})^d)$ be the set of meromorphic sections of $(L^{'})^d.$
Let $D$ be a divisor in $\Gamma(X,(L^{'})^d)$. The truncated counting function to level $k$ of $f$ with respect to $D$ is defined by
               $$N^{[k]}(f,D)= N^{[k]}(f^{*}(D)).$$
\begin{theorem}\label{thm32} (The first main theorem) Let the notations be as above. Then,
                $$T(f,L^{'})=\frac{N(f,D)}{d}.$$ 
\end{theorem}  
\begin{proof} It is easy to see that $\omega_{(L^{'})^d}=d\omega_{L^{'}}$. By definition and by Lemma \ref{lem31}, the conclusion is proved.   
\end{proof}

 Next, we construct the Wronskian. Let $\sigma_0, \sigma_1,\cdots, \sigma_l$ be sections of $L$. Consider a local coordinate $(U,z)$  of $S$ such that $L|_{U}\cong U\times \mathbb{C}$. Assume that 
$\sigma_{jU}$ is the restriction of $\sigma_j$ over $U$. We define
                             
 $$W_{(U,z)}((\sigma_j))=\begin{vmatrix}
                         \sigma_{0U} &\cdots &\sigma_{lU}\\ 
   \frac{d}{dz}\sigma_{0U} & \cdots & \frac{d}{dz}\sigma_{lU}\\
                           \vdots & \cdots & \vdots\\
                            \frac{d^l}{dz^l}\sigma_{0U}&\cdots & \frac{d^l}{dz^l}\sigma_{lU}                            
                          \end{vmatrix}$$
Let $(U^{'},z^{'})$ be an another local coordinate of $S$ and $\sigma_{jU^{'}}$ be the restriction of $\sigma_j$ over $U^{'}$. Similarly, we also define $W_{(U^{'},z^{'})}((\sigma_j))$. Suppose that $U\cap U^{'}\not= \varnothing$ and $\sigma_{jU}= \xi_{U U^{'}}\sigma_{jU^{'}}$, where $\xi_{UU^{'}}$ is transition function. It is easy to check that 
         $$W_{(U,z)}((\sigma_j))=\xi_{UU^{'}}^{l+1}W_{(U^{'},z^{'})}((\sigma_j)) (\frac{dz^{'}}{dz})^{\frac{l(l+1)}{2}}.$$
Therefore, if we set $W((\sigma_j))(x)=W_{(U,z)}((\sigma_j))(x)$ for each $x\in U,$ then 
   $$ W((\sigma_j))\in H^{0}(S, L^{l+1}\otimes K_{S}^{\frac{l(l+1)}{2}}),$$
where $K_S$ is the canonical line bundle over $S$. We have the following
\begin{lemma}\label{lem32} Let $\sigma_0,\sigma_1,\cdots,\sigma_l$ be in $H^{0}(S,L)$. Then,
\begin{itemize}
\item[(i)] \ $W((\sigma_j)) \in H^{0}(S,L^{l+1}\otimes K_{S}^{\frac{l(l+1)}{2}}).$
\item[(ii)] \  $\sigma_0,\cdots,\sigma_l$ are linearly independent iff $W((\sigma_j))\not\equiv 0.$
\item[(iii)] \ Let $A$ be an $(l+1)\times(l+1)$-matrix such that  $(\tau_o,\tau_1,\cdots,\tau_l)^{t}= A(\sigma_0,\sigma_1,\cdots,\sigma_l)^{t}.$ Then
                $W((\tau_j))=(\det A)W((\sigma_j)).$
\item[(iv)] \  If $\Phi \in \Gamma(S, K_{S}^{\frac{l(l+1)}{2}}),$ then
           $ N((\Phi)_0)= l(l+1)(g-1),$
where $g$ is the genus of $S.$
\end{itemize}
\end{lemma}
\begin{proof} 
The proof of $(i)$ is given above and the proof of $(iii)$ is an usual property of Wronskian. Now, we prove $(ii)$ and $(iv).$\\
ii) Obviously, if  $W((\sigma_j))\not\equiv 0,$ then $\sigma_0,\cdots,\sigma_l$ are linearly independent. Conversely, suppose that $\sigma_0,\cdots,\sigma_l$ are linearly independent. We show that
$\sigma_{0U},\cdots,\sigma_{lU}$ are linearly independent for all $U$. Indeed, suppose that  $\sum_{1\le i \le l}a_i \sigma_{iU}=0,$ where $a_i$ are not all zero. Take $U^{'}$ such that $U^{'}\cap U\not =\varnothing.$ Then $\sigma_{iU}=\xi_{UU^{'}}\sigma_{iU^{'}}$ and $\sum_{1\le i \le l}a_i \sigma_{iU^{'}}=0$ on $U \cap U^{'}$. Since $U \cap U^{'}$ is an open subset of $U^{'}$, we get $\sum_{1\le i \le l}a_i \sigma_{iU^{'}}=0$ on $U^{'}$. This follows that $\sum_{1\le i \le l}a_i \sigma_{i}=0.$ This is a contradiction. Thus, we have $W((\sigma_j))\not\equiv 0$.\\
iv) Take a holomorphic $(1,0)$-form $\alpha$ on $S$. By definition of $K_S$, we can consider $\alpha$ as a section of $K_S$. Hence, $\alpha^{\frac{l(l+1)}{2}}\in H^{0}(S,K_{S}^{\frac{l(l+1)}{2}})$. By the Poincare-Hopf index formula for meromorphic differential, we have $N((\alpha)_0) = 2(g-1)$. Thus, $N((\alpha)^{\frac{l(l+1)}{2}}_0) = l(l+1)(g-1)$. By Lemma \ref{lem31} ii), we have 
                 $$N((\Phi)_0)=N((\alpha)^{\frac{l(l+1)}{2}}_0) = l(l+1)(g-1).$$    
\end{proof}

\section{Some lemmas}

\begin{lemma}\label{le:4.1}
 Let $Y$ be a algebraic variety of $\mathbb{C}P^{m}$ of dimension $u$, $Z$ be an algebraic subset of $Y$ and  $H_0, H_1,\cdots,H_{N}$ ($m\ge u+1$)
 be hyperplanes in $\mathbb{C}P^{m}$ such that $H_0 \cap \cdots \cap H_{N}\cap  Y=\varnothing$. Put $R=\{1,2,\cdots,m\}$. 
Then, $\rank(H_j, 0\le j\le N)\ge u-\dim Z$. 
\end{lemma}
\begin{proof}  See Lemma \ref{le:2.1}.        
\end{proof}
\begin{lemma}\label{le:4.2}
Let $H_0, H_1,\cdots,H_{m}$ be hyperplanes in $\mathbb{C}P^{N}$. Put
$k=\rank(H_j, 0\le j\le m)$. Let $E_0, \cdots, E_m$ be positive real number such that $E_j>1$ for $0\le j\le m$. Then, there are $j_1,\cdots,j_k \in R$ such that $\rank \{H_{j_1},\cdots,H_{j_k}\}=k$ and 
             $$E_0 E_1\cdots E_m\le (E_{j_1}\cdots E_{j_k})^{m-k+2}.$$
\end{lemma}
\begin{proof} 
Since $k=\rank(H_j, 0\le j\le m)$ there exist $k$ hyperplanes $H_{j_1},\cdots,H_{j_k}$ in $\{H_0,\cdots,H_m \}$  such that  $H_l$  is  a linear combination of this $k$ hyperplanes. Hence, we have
                $$H_l=\sum_{i=1}^k a_{li}H_{ji}, \text{ where } a_{li}\in \mathbb{C}.$$
Put $H'_l= \sum_{i=1}^k a_{li}z_{i-1}$ $,(0\le l \le m)$ as hyperplanes in $\mathbb{C}P^{k-1}$. It is easy to see that $H'_0,\cdots, H'_m$ are in $m$-subgeneral in  $\mathbb{C}P^{k-1}$. For $0\le j\le m$, put
                         $$\sigma(j)=\lambda=\min\{ \frac{\rank(P)}{\abs{P}}: P\subset R\}$$
Then, by \cite[Proposition 2]{Toda}, we have \\
i) $\lambda \ge \frac{1}{m-k+2}$.\\
ii) For any $P \subset R$,
                     $$\sum_{j\in P} \sigma(j)\le \rank(P)$$  
Next, by \cite[Proposition 1]{Toda}, we obtain that there are $j_1,\cdots,j_k \in R$ such that $\rank (\{H'_{j_1},\cdots,H'_{j_k}\})=k$ and 
             $$(E_0 E_1\cdots E_m)^{\lambda}\le E_{j_1}\cdots E_{j_k}.$$
Since $\rank (\{H'_{j_1},\cdots,H_{j_k}\})=k,$ it implies that $\rank (\{H_{j_1},\cdots,H_{j_k}\})=k$.  Hence, the conclusion is deduced from the fact that $\lambda \ge \frac{1}{m-k+2}$. 
\end{proof}
\begin{lemma} \label{le:4.3}

Let $H_1,\cdots,H_q$ be hyperplanes in $\mathbb{C}P^{m}$. Put $Q=\{1,2,\\
\cdots,q\}$. Let $u$ be a positive integer. Fix $0\le t \le q-1$. Assume that  $\rank\{H_{i_j}(0\le j\le t)\}\ge u+1$ for each $1\le i_0<i_1<\cdots<i_t\le q.$    Let $W$ be a subspace of $\mathbb{C}P^{m}.$  Then, there are $(m-u)$ hyperplanes $T_{1},\cdots, T_{m-u}$ in $\mathbb{C}P^{m}$ such that the following is satisfied:

For each $R\subset Q$ with  $\abs{R}=t+1$ and  $\rank\{H_{j}, j\in R\}\ge u+1$, we have  $\{H_j, T_i: j\in R, 1\le i\le m-u \}$ are in $(t+m-u)-$subgeneral position and $W \not \subset T_i (1\le i\le m-u).$
\end{lemma}
\begin{proof}
Put $T_i:= a_{0i} x_0+\cdots +a_{mi} x_m$ for $1\le i\le m-u$, where $a_{ij}\in \mathbb{C}$. For $R\subset Q$ with $\abs{R}=t+1,$
we consider  determinants of  all submatrices of degree $(m+1)$ of  the matrix of  the coefficients of $H_j (j\in R)\text{ and } T_i (1\le i\le m-u)$. There are $\binom{t+m-u+1}{m+1}$ such matrices. 

Let $h(T,R)$  be a mapping of  $\mathbb{C}^{(m+1)(m-u)}$ into $\mathbb{C}^{\binom{t+m-u+1}{m+1}}$ which maps $(a_{ki}: 0\le k\le m, 1\le i\le m-u)$ to $\binom{t+m-u+1}{m+1}-$tuple of such determinants. Then, $h(T,R)$ is a holomorphic mapping. Since $\rank(H_{i}, i \in R)\ge u+1,$ we have $h(T,R)\not \equiv 0$. Hence, $h(T,R)^{-1}\{0\}$ is a proper analytic subset of $\mathbb{C}^{(m+1)(m-u)}$. On the other hand, we see that there is a proper analytic set $W'$ of  $\mathbb{C}^{(m+1)(m-u)}$ such that if  $(a_{ij}: 0\le i\le m, 1\le j\le m-u)\not \in W'$ then $W\not \subset T_i (1\le i\le m-u).$ Now, taking $(a_{ij}: 0\le i\le m, 1\le j\le m-u)$ in $\mathbb{C}^{(m+1)(m-u)}- (\cup_{\abs{R}=t+1} h(T,R)^{-1}\{0\} \cup W')$, it implies that $T_j$ have the desired property. This finishes the proof.                                    
\end{proof}
 We also need to use a corollary of lemma on logarithmic derivative in \cite{NW10}. 
\begin{lemma}\label{le:2.4} 
( \cite[Lemma 4.2.9]{NW10}) Let $g$ be a non-constant meromorphic function on $\mathbb{C}$. For $k\ge 1$, we have
                            $$\int_{0}^{2\pi} \log \big \lvert \frac{g^{(k)}}{g}(r e^{i\phi}) \big \lvert d\phi = S(r,g),$$
where $H$ is  the hyperplane bundle of $\mathbb{C}P^1,$ and $S(r,g)$ is a quantity satisfying   for arbitrary $\delta >0$,  $S(r,g)= O(\log T(r,g))+ \delta \log r,$ outside a subset of finite Borel measure.   
\end{lemma}

\section{The proof of Theorem A}
 
 We use  the notations as in Sections $1$ and $2.$ 

Replacing $\sigma_i$ by $\sigma_{i}^{\frac{d}{d_i}}$ if necessary, we may assume that $\sigma_1,\cdots,\sigma_q$ are in $H^0 (X, L^d)$ and $\abs{\abs{\sigma_i}}\le 1$. Put $\sigma_i= \sum_{1\le j \le m+1}a_{ij} c_j,$ where $a_{ij}\in \mathbb{C}$. 
 We define a meromorphic mapping $\Phi: X \rightarrow \mathbb{C}P^{m}$ by $\Phi (x):= [c_1(x):\cdots :c_{m+1}(x)].$ 
Also since $X$ compact,  $Y=\Phi(X)$ is an algebraic variety of $\mathbb{C}P^{m}$. Moreover, by definition of $\rank E$, $Y$ is of dimension $\rank E=u$. Put $F= \Phi \circ f$. Since $\overline{f(\mathbb{C})}\cap B(E)=\emptyset$ and $f$ is non-degenerate with respect to $E$, $F$ is holomorphic curve and linearly non-degenerate. Denote by $H$ the hyperplane bundle of $\mathbb{C}P^{m}$. Put $H_i:= \sum_{1\le j \le m+1}a_{ij} z_{j-1}$, where $[z_0,z_1,\cdots,z_m]$ is the homogeneous coordinate of $\mathbb{C}P^m$. It is easy to see that 
\begin{align}\label{eq:1}
T_f(r,L)= \frac{1}{d} T_F(r, H) \text{ and } N_f(r,R_i)= N_F(r, H_i).
\end{align}                    
Furthermore, we have $D_{j_1}\cap \cdots \cap D_{j_t}= B(E)$ if and only if $H_{j_1}\cap \cdots \cap H_{j_t}\cap Y= \Phi(B(E))$. Denote by $\mathcal{K}$ the set of all subsets $K$ of $\{1,\cdots, q\}$ such that $\abs{K}=s_N+1$ and $\cap_{j\in K} D_j=B(E)$. Then $\mathcal{K}$ is the set of all subsets $K\subset \{1,2\cdots,q\}$ such that $\abs{K}=s_N+1$ and $\cap_{j\in K} H_j \cap Y=\Phi(B(E))$.  
 By Lemma \ref{le:4.1} and Lemma \ref{le:4.3}, there are $(m-u)$ hyperplanes $H_{q+1},\cdots,
 H_{q+m-u+b+1}$ in $\mathbb{C}P^{m}$ such that 
   $$\{H_j, H_{q+i}: j\in R, 1\le i\le m-u+b+1 \}$$
 are in $(s_N+m-u+b+1)-$subgeneral position in  the usual sense, where $R\in \mathcal{K}$. 

Put 
$$\mathcal{K}_1=\{R\subset \{1,2,\cdots,q+m-u+b+1\}: \abs{R}=\rank(R)=m+1\}.$$
Note that since $\overline{f(\mathbb{C})}\cap B(E)=\emptyset$ there exists a constant $C>0$ such that 
\begin{align}\label{ine:3}
           C^{-1}< \sum_{\abs{S}=q-N-1} \prod_{j\in S} \bigg(\frac{\abs{H_j (F(z))}}{\lvert \lvert F(z) \lvert \lvert}\bigg)^{\omega(j)} <C 
\end{align}
for all $z\in \mathbb{C}$. Take $R= Q-S$. Put $\omega(j)=\Theta$ for all $j>q$. Then,  
\begin{align}\label{eq:4}
\prod_{j\in S} \bigg(\frac{\abs{H_j (F(z))}}{\lvert \lvert F(z) \lvert \lvert} \bigg)^{\omega(j)}= \bigg(\prod_{j\in R}\frac{\lvert \lvert F(z) \lvert \lvert}{\abs{H_j (F(z))}} \bigg)^{\omega(j)}. \frac{\prod_{j\in Q}\abs{H_j(F(z))}^{\omega(j)}}{ \lvert \lvert F(z) \lvert \lvert^{(\sum_{j\in Q} \omega(j))}}
\end{align}

By Theorem \ref{athe1} and (\ref{eq:4}), for $S$, $\abs{S}=q-N-1$  there exists  $R^0 \in \mathcal{K}$ such that   
\begin{align}\label{ine:5}
\prod_{j\in S} \bigg(\frac{\abs{H_j (F(z))}}{\lvert \lvert F(z) \lvert \lvert} \bigg)^{\omega(j)}\le \bigg(\prod_{j\in R^0}\frac{\lvert \lvert F(z) \lvert \lvert}{\abs{H_j (F(z))}} \bigg). \frac{\prod_{j\in Q}\abs{H_j(F(z))}^{\omega(j)}}{ \lvert \lvert F(z) \lvert \lvert^{(\sum_{j\in Q} \omega(j) )}} .
\end{align}
Hence,
\begin{align}\label{ine:5}
\prod_{j\in S}\frac{\abs{H_j (F(z))}^{\omega(j)}}{\lvert \lvert F(z) \lvert \lvert^{\omega(j)}} &\le \bigg(\prod_{j\in R^0}\frac{\lvert \lvert F(z) \lvert \lvert}{\abs{H_j (F(z))}}\cdot
\prod_{j=q+1}^{q+m-u+b+1}\frac{\lvert \lvert F(z) \lvert \lvert}{\abs{H_j (F(z))}}  \bigg). \\
\nonumber
&\frac{\prod_{j=1}^{q+m-u+b+1}\abs{H_j(F(z))}^{\omega(j)}}{ \lvert \lvert F(z) \lvert \lvert^{(\sum_{j\in Q} \omega(j) +m-u+b+1)}} .
\end{align}
And by Lemma \ref{le:4.2},  there exists  $R^0_1 \subset R^0 \cup \{q+1,\cdots,q+m-u+b+1\} $ such that $\rank(H_j, j\in  R^0_1)=\abs{R^0_1}=m+1$ and 
\begin{align}\label{ine:8}            
\prod_{j\in S}\frac{\abs{H_j (F(z))}^{\omega(j)}}{\lvert \lvert F(z) \lvert \lvert^{\omega(j)}}
                           \le \prod_{j\in R^0_1}\frac{\lvert \lvert F(z) \lvert \lvert ^{s_N-u+2+b}}{\abs{H_j (F(z))}^{s_N-u+2+b}}\cdot \frac{\prod_{j=1}^{q+m-u+b+1}\abs{H_j(F(z))}^{\omega(j)}}{ \lvert \lvert F(z) \lvert \lvert^{(\sum_{j\in Q} \omega(j) +m-u+b+1)}}
\end{align}
By $\rank( R^0_1)=\abs{R^0_1}=m+1$, the Wronskian $W(H_j \circ F, j\in R^0_1)=c(R^0_1)W(F)$ is not identically $0$, where $c(R^0_1)$ is non-zero complex number. Therefore, by (\ref{ine:8}),         
\begin{align}\label{ine:9}
\prod_{j\in S} \bigg(\frac{\abs{H_j (F(z))}}{\lvert \lvert F(z) \lvert \lvert} \bigg)^{\omega(j)} \le 
              \bigg(\frac{\abs{W(H_i\circ F (i\in R^0_1)}}{\prod_{i\in R^0_1}\abs{H_{i} (F(z))}} \bigg)^{s_N-u+2+b}\times
 \\ \nonumber
 \frac{\prod_{j=1}^{q+m-u+b+1}\abs{H_j(F(z))}^{\omega(j)}}{ \lvert \lvert F(z) \lvert \lvert^{(\sum_{j\in Q}\omega(j)-(m+1)(s_N-u+2+b)+m-u+b+1 )} \abs{W(H_i \circ F (i\in R^0_1)}^{s_N-u+2+b}}
 \\ \nonumber
    \le 
              \bigg(\frac{\abs{W(H_i \circ F(i \in R^0_1 )}}{\prod_{i\in R^0_1}\abs{H_{i} (F(z))}} \bigg)^{s_N-u+2+b}\times\\
\nonumber    \frac{\prod_{j=1}^{q+m-u+b+1}\abs{H_j(F(z))}^{\omega(j)}}{ \lvert \lvert F(z) \lvert \lvert^{(\sum_{j\in Q} \omega(j)-(m+1)(n-u+1)+m-u )} \abs{c(R^0_1)W(F)(z)}^{s_N-u+2+b}}.                           
\end{align}

Combining (\ref{ine:9}) with (\ref{ine:3}) we gain that there is a positive real number $C'$ such that
\begin{multline} \label{ine:10}
C' \lvert \lvert F(z) \lvert \lvert^{(\sum_{j\in Q} \omega(j)-(m+1)(s_N-u+2+b)+m-u+1+b )}\le\\
\biggl [ \sum_{R^0_1\in \mathcal{K}_1} \big(\frac{\abs{W(H_i \circ F (i\in R^0_1)}}{\prod_{j\in R^0_1}\abs{H_i (F(z))}} \big)^{s_N-u+2+b} \biggr] \times \\
\frac{\prod_{j=1}^{q+m-u+b+1}\abs{H_j(F(z))}^{\omega(j)})^{s_N-u+2+b}}{\abs{W(F)(z)}^{s_N-u+2+b}}.  
\end{multline}          
Taking the logarithm of (\ref{ine:10}) and applying $2\ddc$ as currents, by the Poincare-Lelong formula (cf. \cite[Theorem 2.2.15, p.46]{NW10}) and the Jensen formula (cf. \cite[Lemma 2.1.30, p.36]{NW10}) and Lemma \ref{le:2.4}, we see that           
\begin{align}\label{ine:11}
 (\Theta (q-k_N)+t_N- (m+1)(s_N-u+1)+m-u) T_F(r,H) \le \\
\nonumber
 \sum_{j=1}^{q+m-u} \omega(j) N_F(r,H_j)
-(s_N-u+1)N(r, \nu_{W(F)})+ A,
\end{align}  
where 
   $$A=\frac{1}{2\pi} \int_{\abs{z}=r}\log\sum_{R\in \mathcal{K}_1}\frac{\abs{W(H_i, i\in R)}}{\prod_{i\in R}\abs{H_{i} (F)}}d\phi +O(1).$$ 
We now prove that 
\begin{align}\label{ine:12}
 \sum_{1\le j\le q+m-u+b+1} \omega(j)(\nu_{H_j(F)}- \nu^{[m]}_{H_j(F)}) \le (s_N-u+2+b)\nu_{W(F)}.    
\end{align}  
By Theorem \ref{athe1}, for any $z\in S$ and for any  $J$  with $\abs{J}=N+1$, there exists a subset $K^{'}(J,z)\in \mathcal{K}$ such that 
\begin{align*}
\sum_{j\in J}\omega(j) (\nu_{H_j(F)}(z)- \nu_{H_j(F)}^{[m]}(z))& \le \sum_{j\in K^{'}(J,z)}(\nu_{H_j (F)}(z)- \nu_{H_j (F)}^{[m]}(z))\\
& \le \max_{K\in \mathcal{K}} \sum_{j\in K}(\nu_{H_j (F)}(z)- \nu_{H_j (F)}^{[m]}(z)) .               
\end{align*}    
Hence, 
\begin{multline}\label{ine:13}
\max_{\abs{J}=N+1} \sum_{j\in J}\omega(j) (\nu_{H_j (F)}(z)- \nu_{H_j (F)}^{[m]}(z)) \le \\ 
   \max_{K\in \mathcal{K}} \sum_{j\in K}(\nu_{H_j (F)}(z)- \nu_{H_j (F)}^{[m]}(z)).
\end{multline}   
On the other hand, we have
\begin{align}\label{ine:14}
\sum_{j=1}^{q}\omega(j) (\nu_{H_j (F)}(z)- \nu_{H_j (F)}^{[m]}(z))=\max_{\abs{J}=N+1} \sum_{j\in J}\omega(j) (\nu_{H_j (F)}(z)- \nu_{H_j (F)}^{[m]}(z))
\end{align}
Put $LH=\sum_{j=1}^{q+m-u+b+1}\omega(j) (\nu_{H_j (F)}(z)- \nu_{H_j (F)}^{[m]}(z)).$ Combining (\ref{ine:14}) and (\ref{ine:13}), by Lemma \ref{le:4.2}, we have
\begin{align}\label{ine:15}
LH & \le \max_{K\in \mathcal{K}} \sum_{j\in K}(\nu_{H_j (F)}(z)- \nu_{H_j (F)}^{[m]}(z))+ 
\sum_{j=q+1}^{q+m-u+b+1}(\nu_{H_j (F)}(z)- \nu_{H_j (F)}^{[m]}(z))
\\
\nonumber
   &\le \max_{R\in \mathcal{K}_1}(s_N-u+2+b)(\sum_{j\in R}(\nu_{H_j (F)}(z)- \nu_{H_j (F)}^{[m]}(z)).
\end{align} 
On the other hand, for $R\in \mathcal{K}_1$  it is well-known that
\begin{align}\label{ine:16}
         \nu_{W(F)}=\nu_{W( H_i (F), i\in R)(z)} \ge (\sum_{j\in R}(\nu_{H_j (F)}(z)- \nu_{H_j (F)}^{[m]}(z)))
\end{align}            
Since (\ref{ine:16}) holds for all $R\in \mathcal{K}_1,$ we have
\begin{align}\label{ine:17}
                     \nu_{W(F)} \ge  \max_{R\in \mathcal{K}_1}(\sum_{j\in R}(\nu_{H_j (F)}(z)- \nu_{H_j (F)}^{[m]}(z))) .
\end{align}
By combining (\ref{ine:17}) and (\ref{ine:16}), we get (\ref{ine:12}). Now, by (\ref{ine:12}), we have
\begin{align}\label{ine:18}
       (s_N-u+2+b)N(r, \nu_{W(F)})\ge  \sum_{j=1}^{q+m-u+b+1} \omega(j)(N_{F}(r,H_j)- N^{[m]}_{F}(r,H_j)) .         
\end{align}
On the other hand, Lemma \ref{le:2.4} yields that
\begin{align}\label{ine:19}
  \frac{1}{2\pi} \int_{\abs{z}=r}\log\sum_{R\in \mathcal{K}_1}\frac{\abs{W(H_i, i\in R)}}{\prod_{i\in R}\abs{H_{i} (F)}}d\phi=S(r,F).    
\end{align}
Furthermore, we get
\begin{align}\label{ine:20}
     N_F(r,H_i)\le T_F(r,H)+0(1),\, q+1\le i\le q+m-u+b+1.
\end{align}            
Combining (\ref{ine:18}), (\ref{ine:19}), (\ref{ine:20}) and (\ref{ine:11}), we have
\begin{multline}\label{ine:21}
 (\Theta (q-k_N)+t_N- (m+1)(s_N-u+2+b)) T_F(r,H)\le \\
\sum_{j=1}^{q} \omega(j) N_F(r,H_j)- \sum_{1\le j\le q} \omega(j)(N_{F}(r,H_j)- N^{[m]}_{F}(r,H_j))+ S(r,F).
\end{multline}
Hence, 
\begin{multline}\label{ine:22}
 (\Theta (q-k_N)+t_N- (m+1)(s_N-u+2+b)) T_F(r,H) \le \\
 \omega(j)\sum_{1\le j\le q} N^{[m]}_{F}(r,H_j) +S(r,F).
\end{multline}
Remark that $\omega(j)\le \Theta$ and $\Theta \ge \frac{t_N}{k_N}$. By dividing two sides of (\ref{ine:22}) by $\Theta,$ we see that  
\begin{multline}\label{ine:23}
 ((q-k_N)+k_N- \frac{k_N(m+1)}{t_N}(s_N-u+2+b)) T_F(r,H) \le \\
 \sum_{1\le j\le q} N^{[m]}_{F}(r,H_j) +S(r,F),
\end{multline} 
Combining (\ref{ine:23}), (\ref{ine:22}) and (\ref{eq:1}), the proof of Theorem A is completed. $\square$

\vskip0.2cm
\begin{remark}\label{re50}\ By the hyperthesis, in Theorem A, we have
$$\liminf_{r\to \infty} \dfrac{T_f(r,L)}{\log r} >0.$$
\end{remark}
Indeed, let $F: \mathbb{C} \rightarrow \mathbb{C}P^m$ be as in the proof of Theorem A. Then we have 
             $$T_f(r,L)= \frac{1}{d} T_F(r, H).$$
Suppose that 
     $$\liminf_{r\rightarrow \infty}\frac{T_f(r,L)}{\log r}=0.$$ 
Then   
    $$\liminf_{r\rightarrow \infty}\frac{T_F(r,H)}{\log r}=0.$$ 
Hence, $F$ is constant. This implies that there are $a_i  \in \mathbb{C}(0\le i \le m)$ such that  $c_i(f)=a_i c_0(f),$ where $c_i$ are as in the proof of Theorem A.  Therefore, we have $f(\mathbb{C})\subset \{c_i -a_i c_0=0\}$ which is a hypersurface defining by a section  of $E.$ This is a contradiction. It follows that   
          $$\liminf_{r\rightarrow \infty}\frac{T_f(r,L)}{\log r}>0.$$  

By using Theorem A, we have the following Ramification Theorem.
\vskip0.2cm
\begin{corollary}\label{co51}\ Let $X$ be a compact complex manifold. Let $L\rightarrow X$ be a holomorphic line bundle over $X$. Fix a positive integer $d$.  Take positive divisors $d_1, d_2,\cdots, d_q$  of $d$. Let $\sigma_j \ (j=1,2,\cdots,q)$ be in $H^{0}(X, L^{d_j}).$ Let $E$ be  the  $\mathbb{C}$-vector subspace of $H^{0}(X, L^{d})$ generated by       $\sigma_{1}^{\frac{d}{d_1}},\cdots,\sigma_{q}^{\frac{d}{d_q}}$. Put  $u=\rank E, \dim E=m+1$ and $b=\dim B(E)+1$ if $B(E)\not =\varnothing$, otherwise $b=-1$.
Set $D_j=\{\sigma_j=0\}$ and denote by $R_j$ the zero divisors of $\sigma_j \ (j=1,2,\cdots,q)$.  Assume that $D_1,\cdots,D_q$ are in $N$-subgeneral position with respect to $E$ and $u>b.$   Let $f: \mathbb{C} \rightarrow X$ be an analytically non-degenerate holomorphic mapping with respect to $E$, i.e  $f(\mathbb{C})\not\subset \supp(\nu_{\sigma})$ for any $\sigma \in E\setminus \{0\}$ and $\overline{f(\mathbb{C})}\cap B(E)=\varnothing.$  Assume  that 
$f^*R_j \geq v_j \text{supp}\,f^*R_j\ (1\le j\le q),$ where $v_j$ is a positive integer if $f^*R_j\not=\varnothing$ and $v_j=\infty$ if
$f^*R_j=\varnothing.$ 
Then, 
   $$ \sum_{j=1}^q (1- \min\{1, \frac{m}{v_i}\}) \le (m+1)K(E,N,\{D_j\}),$$
where $K(E,N, \{D_j\})$ is as Theorem A.
\end{corollary}

\begin{proof} \  In order to prove Corollary \ref{co51} we can assume that   
$$f^*R_j\not=\varnothing\ (1\le j\le q).$$
If $v_i \ge m,$ we have
\begin{align*}
1- \frac{N^{[m]}_f(r,R_i)}{d_i \cdot T_f(r,L)}&\ge 1- \frac{N^{[m]}_f(r,R_i)}{N_f(r,R_i)}\\
                                                                        &\ge 1- \frac{m}{v_i}. 
\end{align*}
If $v_i \le m,$ we have 
  $$1- \frac{N^{[m]}_f(r,R_i)}{d_i \cdot T_f(r,L)}\ge 1- \frac{N^{[m]}_f(r,R_i)}{N_f(r,R_i)}\ge 0$$
Hence, we get 
       $$
1- \frac{N^{[m]}_f(r,R_i)}{d_i \cdot T_f(r,L)}\ge 1- \min\{1, \frac{m}{v_i}\}.$$
This implies the conclusion. 
\end{proof}  
\section{The proof of Theorems B}
  The proof of Theorem B is essentially similar to the one of Theorem A. 

Put $\widetilde{L}=f^{*} L.$
 Let the notations be as in the proof of Theorem A in which the complex plane $\mathbb{C}$ is replaced by a compact Riemann surface $S.$ We have
\begin{align}\label{eq:26}
T(f,L)= \frac{1}{d} T(F, H) \text{ and } N(f,R_i)= N(F,H_i).
\end{align}  
 
By  similar arguments to the ones in the proof of Theorem B, we get 
\begin{align}\label{ine:27}
 (\Theta (q-k_N)+t_N- (m+1)(s_N-u+2+b)+m-u+b+1) T(F,H) &\le \\ 
\nonumber
\sum_{j=1}^{q+m-u+b+1} \omega(j) N(F,H_j)
-(n-u+2+b)N(r, \nu_{W(F)})+ A
\end{align}  
and 
\begin{align}\label{ine:28}
       (s_N-u+2+b)N(\nu_{W(F)})\ge  \sum_{1\le j\le q} \omega(j)   ( N (F,H_j)- N^{[m]}(F,H_j)),        
\end{align} 
where $\omega(j)=\Theta$ if $j>q$
and 
$$A=\frac{1}{2\pi} \int_{\abs{z}=r}\log\sum_{R\in \mathcal{K}_1}\frac{\abs{W(H_i, i\in R)}}{\prod_{i\in R}\abs{H_{i} (F)}}d\phi.$$
Obviously, $H_j(F)$ are sections of $L^d$. By Lemma \ref{lem32},  we have
\begin{align}\label{eq:29}
              W(F) \in H^{0}(S, \widetilde{L}^{d(m+1)} \otimes K_S^{m(m+1)/2}) 
\end{align}
Moreover, for each $ R\in \mathcal{K}_1$, it implies that
\begin{align}\label{eq:30}
    \prod_{j\in R} H_{j}(F)\in H^0 (S,\widetilde{L}^{d(m+1)})        
\end{align}
Combining (\ref{eq:30}) and (\ref{eq:29}), we get
\begin{align}\label{eq:31}
         \frac{W(F)}{\prod_{j\in R} H_{j}(F)} \in H^0(S,K_S^{m(m+1)/2}),
\end{align}
where   $ R\in \mathcal{K}_1$ and $K_S$ is the canonical bundle of $S$.
Hence, by Lemma \ref{lem31}, we have
\begin{align}\label{eq:32}
\int_{S}2\ddc \log \sum_{R\in \mathcal{K}_1}\frac{\abs{W(H_i (i\in R),T)}}
{\prod_{j\in R}\abs{H_{j} (F)}. \prod_{i=1}^{m-u+b+1}\abs{T_i}}=m(m+1)(g-1).    
\end{align}   
Combining (\ref{eq:32}), (\ref{ine:28}) and (\ref{ine:27}), we get
\begin{multline}\label{ine:33}
(\Theta (q-k_N)+t_N- (m+1)(s_N-u+2+b)) T(F,H)\le \\
           \sum_{j=1}^{q} \omega(j) N^{[m]}(F,H_j)+ m(m+1)(g-1).
\end{multline}
Dividing by $\Theta$ two sides of (\ref{ine:33}) and noting that $\Theta \ge \omega(j)$, $\frac{t_N}{k_N} \le \Theta$,  it implies that  
\begin{multline}\label{ine:34}
(q-k_N+ k_N(1- \frac{m+1}{t_N}(s_N-u+2+b)) T(F,H)\le \\
           \sum_{j=1}^{q} N^{[m]}(F,H_j)+ A(d,L)
\end{multline}
where $A(d,L)=\begin{cases}
 \frac{m(m+1)k_N(g-1)}{n+1} & \text{ if } g\ge 1\\
  0 & \text{ if } g=0. 
\end{cases}$\\
By (\ref{eq:26}), (\ref{ine:34}), the proof of Theorem B is completed.
            
\section{Applications}

\subsection{Unicity theorem}\

\vskip0.2cm
Applying Theorem B, we also get  a unicity theorem for holomorphic curves of  a compact Riemann surface
 into a compact complex manifold sharing divisors in $N$-subgeneral position in this manifold.
Namely, we get the following.

\vskip0.2cm
\begin{proposition}\label{apro73}\ Let $S$ be a compact Riemann surface with genus $g$ and $X$ be a compact complex manifold of dimension $n$. 
Let $L\rightarrow X$ be a holomorphic line bundle over $X$. Fix a positive integer $d$. 
Let $E$ be  a  $\mathbb{C}$-vector subspace of dimension $m+1$  of $H^{0}(X, L^{d})$.  
Take positive divisors $d_1, d_2,\cdots, d_q$  of $d$. 
Let $\sigma_j (j=1,2,\cdots,q)$ be in $H^{0}(X, L^{d_j})$ such that $\sigma_{1}^{\frac{d}{d_1}},\cdots,\sigma_{q}^{\frac{d}{d_q}}\in E$. 
Denote by $R_j$ the zero divisors of $\sigma_j$.  
Assume that $R_1,\cdots,R_q$ are in $N$-subgeneral position in $X$ and $\rank E=n.$ 
Let $f_1,f_2: S\rightarrow X$ be a holomorphic mapping such that $f_i(i=1,2)$ is analytically nondegenerate with respect to $E$, i.e  $f_i(S)\not\subset \supp(\nu_{\sigma})$ for any $\sigma \in E\setminus \{0\}$ and $f_i(\mathbb{C})\cap B(E)=\varnothing, (i=1,2).$
Assume that  three the following conditions are satisfied.
\begin{itemize}
\item[(i)] \ $\bigcup_{i-1}^q f_1^{-1}(\supp R_i)\not = \varnothing$  and  $f_1^{-1}(\supp R_i)=f_2^{-1}(\supp R_i)$ for each $1\le i \le q,$ 
\item[(ii)] \ $f_1=f_2 $ on $\bigcup_{1\le i\le q} f_1^{-1}(\supp R_i),$
\item[(iii)] \ $f_1,f_2$ are analytically non-degenerate holomorphic mappings with respect to $H^0(X,K^{\binom{m}{n}}_X \otimes L^{2n \binom{m}{n}}),$ i.e  $f_i(S)\not\subset \supp(\nu_{\sigma})$ for any $\sigma \in H^0(X,K^{\binom{m}{n}}_X \otimes L^{2n \binom{m}{n}})\setminus \{0\},$ where $K_X$ is the canonical bundle of $X.$
\end{itemize}
Then, $f_1\equiv f_2$ for each $q > B(d,L)+ 2 m(N+1)+2A(d,L),$  where $k_N,s_N,t_N \text{ are defined as in Proposition \ref{apro8}},$
$A(d,L)$ is as in Theorem B and 
$B(d,L)= \dfrac{k_N(m+1)}{t_N}.$
\end{proposition}

\begin{proof}
Let $\Phi$ be the mapping which is defined in the proof of Theorem B.  Put $F_i=\Phi\circ f_i$ for $i=1,2$.  Denote by $L_i$ the pullback $f_i^* L$ of $L$ by $f_i$  ($i=1,2$). Let $H_m$ be the hyperplane bundle of $\mathbb{C}P^m$. It is easy to see that $F_i^*H_m=L_i \ (i=1,2).$  Take two distinct hyperplanes $\Xi, \Xi '$ in $\mathbb{C}P^{m}$. Then, $F_i^*\Xi,F_i^*\Xi'$ are sections of $L_i$. It implies that $F_1^*\Xi  \otimes F_2^*\Xi', F_1^*\Xi ' \otimes F_2^*\Xi$ are sections of $L_1 \otimes L_2.$ 
Put $h= F_1^*\Xi  \otimes F_2^*\Xi'/F_1^*\Xi ' \otimes F_2^*\Xi$. Then, $h$ is a meromorphic function on $\mathbb{C}$ and consider $h$ as a mapping of $\mathbb{C}$ into $\mathbb{C}P^1$ by
                      $$h(z)=[ F_1^*\Xi  \otimes F_2^*\Xi'(z): F_1^*\Xi ' \otimes F_2^*\Xi(z)]$$ 
for all $z\in S.$ Suppose that $h$ is nonconstant. Let $H_1$ be the hyperplane bundle of $\mathbb{C}P^1$. Then, by Theorem \ref{thm32}, we have 
\begin{align}\label{ine:36}
T(h, H_1)=N(h,[1:0]) & \le N(F_1, \Xi)+ N(F_2, \Xi')\\
\nonumber                                  &= T(F_1, H_m)+ N(F_2,H_m)\\
\nonumber                                  &= T(f_1, L)+ T(f_2, L). 
\end{align}
Applying Theorem B to $f_1$ and $f_2$, we get
\begin{multline}\label{ine:37}
(q-B(d,L))(T(f_1,L)+ T(f_2,L))\le  \\
                     \sum_{j=1}^{q} (N^{[m]}(f_1,R_j)+ N^{[m]}(f_2,R_j))+ 2A(d,L).
\end{multline}
Put $M=\bigcup_{1\le i\le q} f_1^{-1}(D_i)$. Since $f_1=f_2$ on $M,$ it implies that 
\begin{align}\label{ine:38}
\abs{M}\le N(h, [1:1])=T(h,H_1).  
\end{align}                     
Combining (\ref{ine:36}),(\ref{ine:37}) and (\ref{ine:38}), we obtain
\begin{align}\label{ine:39}
(q-B(d,L))\abs{M}\le \sum_{j=1}^{q} (N^{[m]}(f_1,R_j)+ N^{[m]}(f_1,R_j))+ 2 A(d,L).
\end{align}
On the other hand, by the property of $N-$subgeneral position, there is no point $x\in S$ such that $f_1(x)$ belong to $(N+1)$ hypersurfaces $D_j$. Therefore, 
                $$(N+1)\abs{M}\ge \sum_{j=1}^{q} N^{[1]}(f_1,R_j)$$  
Furthermore, we have
                 $$ \sum_{j=1}^{q} N^{[1]}(f_1,R_j)\ge \frac{1}{m}\sum_{j=1}^{q} N^{[m]}(f_1,R_j). $$
Combining the two above inequalities, we get
\begin{align*}
                m(N+1)\abs{M} \ge \sum_{j=1}^{q} N^{[m]}(f_1,R_j).
\end{align*}                 
Since $f_1^{-1}(D_j)=f_2^{-1}(D_j),$  $N^{[1]}(f_1,R_j)=N^{[1]}(f_2,R_j)$. Hence, we have
\begin{align}\label{ine:40}
                2 m(N+1)\abs{M} \ge \sum_{j=1}^{q} (N^{[m]}(f_1,R_j)+N^{[m]}(f_2,R_j)).
\end{align}
By (\ref{ine:39}) and (\ref{ine:40}), we obtain
\begin{align*}
       (q-B(d,L))-2 m(N+1))M\le 2 A(d,L).
\end{align*}
Since $q>  B(d,L)+2 m(N+1)+2A(d,L)$, we get a contradiction. Hence, $h$ is constant. It follows that $F_1 =F_2.$ \\
 
 Put $T=\{x\in X:  \rank_x \Phi <n\}$, $U_i=\{\mathbf{z} \in  \mathbb{C}P^{m}: z_i \not =0\}, i=1,2,\ldots,m+1.$ For each $x \in T,$ take a local coordinate $(x_1,\ldots,x_n)$ of $X$ around $x$  and $U_i$ such that $c_i(x) \not =0.$ Then the Jacobian of $F$ is
$$ \begin{pmatrix}
(\frac{c_1}{c_i})'_{x_1} & (\frac{c_1}{c_i})'_{x_2}& \ldots & (\frac{c_1}{c_i})'_{x_n}\\
\vdots & \ldots & \cdots \\
(\frac{c_{m+1}}{c_i})'_{x_1} & (\frac{c_{m+1}}{c_i})'_{x_2}& \ldots & (\frac{c_{m+1}}{c_i})'_{x_n}
\end{pmatrix} $$
 Since $x\in T$, every minor determinant of  degree $n$ of the above matrix equals  zero at $x$. Hence, $T$ is in the union of zero sets of such minor determinants. We see that there are $\binom{m}{n}$ submatricies of degree $n$ of the above matrix. Let $J_0(x_1,\ldots,x_n)$ be the one of these submatricies. If $(x'_1,\ldots,x'_n)$ is an another coordinate of $X$ around $x$, then we have
\begin{align}\label{ine:46}
\abs{J_0(x_1,\ldots,x_n)}=\abs{J_0(x'_1,\ldots,x'_n)}.\bigg |{\frac{\partial(x'_1,\ldots,x'_n)}{\partial (x_1,\ldots,x_n)}\bigg |}.
 \end{align}
By calculating the determinant of $J_0$, we get 
\begin{align}\label{ine:47}
\abs{J_0}= \frac{\alpha(c_k, (c_l)')}{c_i^{2n}}.
  \end{align}
Since $c_i^{2n} \in H^0(X, L^{2n})$ and by (\ref{ine:46}) and (\ref{ine:47}), we get $\alpha(c_k, (c_l)') \in H^0(X, K_X \otimes L^{2n})$. Denote by $\alpha$ the product of all such quantities. Then, 
 $$\alpha \in H^0(X, K^{\binom{m}{n}}_X \otimes L^{2n \binom{m}{n}}). $$
Since $f_1$ and $f_2$ are  analytically non-degenerate holomorphic mappings with respect to $H^0(X,K^{\binom{m}{n}}_X \otimes L^{2n \binom{m}{n}}),$ it implies that  $f_1(X) \not \subset T$ and $f_2(X) \not \subset T.$ On the other hand,  since $\Phi$ has maximal rank in $X-T$,   $\Phi$ is a covering of $X-T$ onto $\Phi(X)- \Phi(T)$. By the lifting theorem, we get $f_1=f_2$ on $\mathbb{C}- (f_1^{-1}(T)\cup f_2^{-1}(T))$ and hence, $f_1=f_2$ on $\mathbb{C}$. The proof is finished.  
\end{proof}
Finally, we construct an example to show that  the condition $\rank E=n$ cannot omit in Proposition \ref{apro73}.
\begin{example} \label{aex1}
For each $l\in \mathbb{N},$ denote by $H_l$  the hyperplane bundle of $\mathbb{C}P^l$. Let $m,k$ be the fixed integers. Put $X= \mathbb{C}P^m \times \mathbb{C}P^k$. Let $\{U\}$ be an open  cover of $\mathbb{C}P^m$ and $\{\lambda_{UV}\}$ be the transition function system of $H_m$ corresponding to the cover $\{U\}$. Consider the family $\{U^{*}= U\times \mathbb{C}P^k\}$. Put $\lambda_{U^* V^*}(x,y)=\lambda_{UV}(x)$ for each $x\in U\cap V, y \in \mathbb{C}P^k.$
Hence, there exists a line bundle $L^*$ over $X$ such that $\{\lambda_{U^* V^*}\}$ is  its transition function system. Take a section $\sigma^*$ of $L^*$. By the compactness of $\mathbb{C}P^k$, it implies that there is a section $\sigma$ of $L$ such that $\sigma^*(x,y) =\sigma(x).$ Hence, each divisor of  a section of $L^*$ is the Cartesian product of a divisor of $L$ and $\mathbb{C}P^k$.  It is easy to check that $\rank H^0 (X,L) \le m < m+k=$ dimension of $X$.  We can choose a Riemann surface $S$ and divisors $D_1,\ldots,D_q$ of $L$ such that there exist  holomorphic mapping $f$ of  $S$ to $\mathbb{C}P^m$ and two holomorphic mapping $g_1,g_2$ of $S$ into $\mathbb{C}P^k$ which satisfy $g_1=g_2$ on $\cup_{j=1}^{q} f^{-1}(D_j)$ and $g_1 \not \equiv g_2$. 
By a direct computation, we see that  $H^0(X, K_{X}^l \otimes L^t)=0$ for all $l,t>0$. 
Therefore, $g_1,g_2$ is  non-degenerate with respect to $H^0(X, K_{X}^l \otimes L^t)$ for all $l,t>0$.  Define mappings $f_i (i=1,2)$ of $S$ into $X$ by $f_i=(f,g_i)$. Obviously, $f_1,f_2$  satisfy the  conditions $(i), (ii), (iii)$ in Proposition \ref{apro73}, but they are distinct.     
\end{example}

\subsection{Five-Point Theorem of Lappan in high dimension}\

We now recall the following Five-Point Theorem of Lappan \cite{Lap}.

{\bf Theorem of Lappan}\ (see \cite{Lap}).\ {\it Let $A$ be a subset of $\mathbb{C}P^{1}$ with at least $5$ elements.
Then $f \in Hol(\Delta,\mathbb{C}P^{1})$ is normal iff
$$ \sup \big\{ |f'(z)| (1-|z|^2)/(1+|f(z)|^2) :
z \in f^{-1} (A) \big\} < \infty ,$$
where $\Delta$ is the open unit disc in $\mathbb{C}$.}

We now extend this theorem of Lappan 
to a normal family from an arbitrary hyperbolic complex manifold to a compact
complex manifold. First of all, we recall some notions.

Let $Hol(X,Y)$ $(\mathcal{C}(X,Y))$ represent the family of
holomorphic (continuous) maps from a complex (topological) space
$X$ to a complex (topological) space $Y$, and let $Y^+ = Y \cup \{\infty\}$
be the Alexandroff one-point compactification of $Y$ if $Y$ is not
compact, $Y^+ = Y$ if $Y$ is compact. The topology used on all function 
spaces is the compact-open topology.

\begin{definition}\label{def7.1} (see \cite[p.348]{JK})

Let $X$, $Y$  be complex spaces and let $\mathcal{F} \subset Hol(X,Y)$

i) A Brody sequence for $\mathcal{F}$ is a sequence $\{f_n \circ g_n\}$, 
where $f_n \in \mathcal{F}$ and $g_n \in Hol(\Delta_n,X)$, where 
$\Delta_n = \{ z \in \mathbb{C} : |z| < n\}$.

ii) A map $h \in \mathcal{C}(\mathbb{C}, Y^+)$ is a Brody limit for $\mathcal{F}$
if there is a 
Brody sequence $\{h_n\}$ for $\mathcal{F}$ such that $h_n \to h$ on the
compact subsets of $\mathbb{C}$.
\end{definition}

\begin{definition}\label{def7.2} (see \cite[p.348]{JK})

We say that a family $\mathcal{F}$ of holomorphic mappings from a complex
space $X$ to a complex space $Y$ is uniformly normal if 
$$\mathcal{F} \circ Hol(M,X) = \big\{ f \circ g : f \in \mathcal{F},\
g \in Hol(M,X)\big\}$$
is relatively compact in $\mathcal{C}(M,Y^+)$ for each complex space $M$,
and that $f \in Hol (X,Y)$ is a normal mapping if $\{f\}$ is 
uniformly normal.
\end{definition}

As in \cite{JK}, we have the following assertion.

If $X$, $Y$ are complex spaces, then $\mathcal{F} \subset Hol(X,Y)$ is
uniformly normal iff $\mathcal{F} \circ Hol(\Delta,X)$ is relatively
compact in $\mathcal{C}(\Delta,Y^+)$.  

   Let $X$ be a complex manifold and $J_k(X)$ be the $k$-jet bundle over $X.$ Given a holomorphic mapping $f: \Delta_r \rightarrow X$ with $f(0)=x$, we denote by $j_k(f)$ the element of $J_k(X)_x$  defined by the germ of $f$ at $0.$  

    Let $U$ be an open subset of $\mathbb{C}$ and $f: U \rightarrow X$ be a holomorphic curve. We now define a holomorphic mapping
          $J_k(f):U\rightarrow J_k(X).$
Indeed, for each $z\in U$ with  $w\in U-z=U_z$, we put $f_z(w) :=f(z+w)$. Then $f_z$ is a holomorphic mapping of  a neighborhood $U_z$ of $0$ into $X.$ Set 
                       $$J_k(f)(z)=j_k(f_z).$$

The mapping $J_k(f)$ is said to be  a \emph{$k$-jet lift} of $f.$ 

\begin{definition}\label{def7.3}
Let $X$ be   a complex manifold.
A \emph{$k$-jet pseudo-metric} on $J_k(X)$ 
is a real-valued nonnegative continuous function $F$ defined
on $J_k(X)$ satisfying 
                   $$F(c \xi)= \abs{c} F(\xi)\quad \xi \in J_k(X), c\in \mathbb{C}.$$              
Additionally, if  $F(\xi)=0$ iff $\xi=0,$ then $F$ is called a \emph{$k$-jet metric} on $J_k(X).$ 
\end{definition}
																																																						
\begin{proposition}\label{pro74}
Let $X$ be   a complex manifold of dimension $m$. Then, there exists a $k$-jet metric $F$ on $J_k(X).$
\end{proposition} 
\begin{proof}
Let $\{(U_i, \pi_i)\}$ be a trivialization system of $J_k(X)$ over $X$ such that $\pi: U_i \rightarrow V_i \times \mathbb{C}^{km}$, where  $V_i$ is an open polydisc in $\mathbb{C}^m.$  Take a partition of unity $\{c_i\}$ subordinated to the open covering $\{U_i\}.$ Assume that $(x_1^i,\ldots,x_m^i)$ is a local coordinate system of $X$ on $V_i.$ Define the mapping 
     $ F_i: U_i \rightarrow \mathbb{R}^{+}$ by 
                            $$j_k(f) \longmapsto \sum_{t=1}^k (\sum_{l=1}^m \abs{d^t x^i_l(f)(0)}^{s_t})^{\frac{1}{k!}},$$
where $s_t=\dfrac{k!}{t}\, (1\le t\le k).$ 
Obviously, $F_i$ is a $k$-jet metric on $J_k(V_i).$ Put $F=\sum_i c_i F_i.$ Then $F$ is a $k$-jet metric on $J_k(X).$
\end{proof}
  
  Given a point $x\in X, \xi \in J_k(X)_x$, the Kobayashi $k$-pseudo-metric $K^k_X(x,\xi)$ is determined
  by
$$K^k_X(x,\xi) = \inf \big\{\frac{1}{r} : \varphi(0) = x,
J_k(\varphi)(0) = \xi \\ \text{ for some } \varphi \in Hol(\Delta_r,X)\big\}.$$
 
  For a holomorphic mapping $g$ of $Y$ into $X$, the pull-back  $g^*K^k_X$ of  $K^k_X$ is a pseudo-metric on $Y$ given by 
      $$g^*K^k_X(y, j_k(f))= K^k_X(g(y), j_k(g\circ f)).$$ 

By the above definitions, it is easy to get the following.  
\begin{lemma}\label{le75}
Let the notations be as above. Then,

{\rm (i)} $g^*K^k_X(y, \xi_y ) \le K^k_Y(g(y), g_{*}\xi_y)$ for all $y\in Y, \xi_y\in J_k(Y)_y$ .

{\rm (ii)} $K^k_{\Delta_r}(z,j_k(id))\le \frac{1}{r},$ for all $z\in \Delta_r$ ($id$ is the identity mapping). 
\end{lemma}

\begin{proposition}\label{pro76} (see \cite[Theorem A]{TD}) \ Let $X$ be a complex manifold. Then $X$ is hyperbolic iff for each $x\in X$ and for each open neighborhood $U$ of $x,$ there exist an open neighborhood $V$ of $x$ in $U$ and a positive constant $C$ such that 
                                  $$K^k_U(y,\xi_y) \le C. K^k_X(y,\xi_y),$$
for all $k\ge 1$, for all $y\in V$ and $\xi_y \in J_k(X)_y.$ 
\end{proposition} 
For more fundamental properties of $K^k_X$, see \cite{TD}.       
\begin{proposition}\label{pro77}
  $K^k_{\Delta}(y,\xi_y)>0$ for each $y\in \Delta$ and  $\xi_y\in J_k(\Delta)- \{0_y\}.$ 
\end{proposition}     
\begin{proof}
Assume that  $\varphi: \Delta_r \to \Delta$ is a holomorphic mapping such that $\varphi(0)=y$ and $J_k(\varphi)(0)=\xi_y.$ Since $\xi_y:=(\xi^1_y,\cdots,\xi^k_y)\not = 0_y$, there is $1\le i\le k$ such that $\xi^i_y\not =0$. Since
                          $$\varphi'(z)=\frac{1}{2\pi i}\int_{\partial \Delta_r} \frac{f(a)}{(a-z)^2}da$$
for all $z\in \Delta_r,$ it implies that 
                       $\abs{\varphi'(z)}\le \frac{4}{r^2}$ for each $z\in \Delta_{\frac{r}{2}}.$

Repeating the above argument  for $\varphi', \varphi^{(2)},\ldots,$ we get 
                        $$\abs{\varphi^{(k)}(z)}\le \frac{2^{n^2}}{r^{2k}}$$
for all $z\in \Delta_{\frac{r}{2^k}}.$ Therefore, we have
 $$g=\varphi^{(k-1)}: \Delta_{\frac{r}{2^{k-1}}}\rightarrow \Delta_{\frac{2^{n^2}}{r^{2k}}}$$ 
and  $g'(0)=\xi^i_y.$ This implies that $K^k_{\Delta}(y,\xi_y)\ge K^1_{\Delta_{\frac{2^{n^2}}{r^{2k}}}}(y,\xi^i_y)>0.$ 
We get the desired conclusion.                         
\end{proof}
Combining Proposition \ref{pro76} and Proposition \ref{pro77}, we get              
\begin{corollary}\label{co78}
Let $X$ be a hyperbolic complex manifold. Then, we have $K^k_X(x,\xi_x)>0$ for all $x\in X, \xi_x\in J_k(X)_x-\{0_x\}.$
\end{corollary}  
\begin{proposition}\label{pro79}
Let $f: X\rightarrow Y$ be a holomorphic mapping between complex manifolds such that $f$ is normal. Assume that $F$ is  a $k$-jet metric  on $J_k(Y)$ and  $r>0.$ Then, there exists a constant $c>0$ such that 
                  $$F(J_k(f\circ \phi)) \le c\, \text{ for each }\phi \in Hol(\Delta_r,M).$$
\end{proposition}
\begin{proof}
Suppose the contrast. Then, there exist $\{\phi_n\}\subset Hol(\Delta,M)$ and $\{z_n\}\subset \Delta$ such that $F \circ J_k(f\circ \phi_n) (z_n) \ge n$ for all $n.$ Take  an automorphism $T_n$ of $\Delta$ such that $T_n(0)=z_n.$ Put $\phi'_n= \phi_n \circ T_n.$ Then 
                     $$F \circ J_k(f\circ \phi '_n) (0) \ge n \text{ for all } n. $$   
Since $f$ is normal, there exists $g\in Hol(\Delta_r, X)$ such that $f\circ\phi'_n \rightarrow g.$ Therefore, there exist an open subset $V$ of $\Delta_r$ around $0$ and a local coordinate $U$ of $g(0)$ in $X$ such that $f\circ\phi'_n(z), g(z)\in U$ for all $z\in V.$ This implies that $(f\circ\phi'_n)^{(t)}\rightarrow g^{(t)}$ (the $t$-th derivative) on $V.$ Hence, 
          $$F\circ J_k(g) (0)= \lim_{n\rightarrow \infty} F \circ J_k(f\circ \phi '_n) (0)=\infty.$$
This is a contradiction.                     
\end{proof}
\begin{corollary}\label{co710}
Let $f: X\rightarrow Y$ be a holomorphic mapping between complex manifolds such that $f$ is normal. Assume that $F$ is  a $k$-jet metric  on $J_k(Y)$  and $K^k_X(x, \xi_x)>0\, \text{ for all } x\in X, \xi_x\in J_k(X)-\{0_x\}.$ Then,  there exists a constant $c>0$ such that 
                   $$f^* F(x, \xi_x) \le c\cdot K^k_X(x, \xi_x)\, \text{ for all } x\in X, \xi_x\in J_k(X).$$
\end{corollary}
\begin{proof}
Suppose the contrast. Then, there are $x_n \in X, \xi_n\in J_k(X)_{x_n}$ such that 
$f^* F(x_n, \xi_n) \ge n\cdot K^k_X(x_n, \xi_n).$
Put 
 $$\xi'_n= \frac{\xi_n}{K^k_X(x_n, \xi_n)}.$$
Then, $K^k_X(x_n,\xi'_n)=1.$ By the definition of $K^k_X$, for each $n$, there exists $\phi_n\in Hol(\Delta_{\frac{1}{2}},X)$ such that $J_k(\phi)(0)=\xi_n.$ Hence,  $F(J_k(f\circ \phi_n)) \ge n$ for all $n.$ This is a contradiction.
\end{proof}
\begin{definition}\label{de711}
Let $\Omega$ be a complex manifold and $X$ be a compact complex manifold. Let $L\rightarrow X$ be a holomorphic line bundle over $X$. Let $E$ be  a  $\mathbb{C}$-vector subspace of $H^{0}(X, L)$ of dimension $m+1.$ Let $F_m$ be a $m$-jet metric on $J_m(X).$ Let $f$ be a holomophic mapping of $\Omega$ into  $X$. Assume that $K^m_{\Omega}(p,\xi_p)>0 $ for all $ p\in \Omega, \xi_p\in J_m(\Omega)_p-\{0_p\}.$  For $p\in \Omega$, we put
\begin{align*}\
|df(p)|_{F_m} &= \sup \Big\{ \frac{f^* F_m(p, \xi_p)}{K^m_{\Omega}(p,\xi_p)}: \xi_p \in J_m(\Omega)_p -\{0_p\}\Big\}\\
                      &=\sup \Big\{ f^* F_m\big(p, \xi_p): K^m_{\Omega}(p,\xi_p)=1\Big\}.
\end{align*}          
\end{definition}

We now prove the main theorem of this subsection.
\begin{theorem} \ Let $\Omega$ be a hyperbolic complex manifold  and $X$ be a compact complex manifold. Let $L\rightarrow X$ be a holomorphic line bundle over $X$. Fix a positive integer $d$.  Take positive divisors $d_1, d_2,\cdots, d_q$  of $d$. Assume that $\sigma_j \in H^{0}(X, L^{d_j}), D_j=\{\sigma_j=0\}\ (j=1,2,\cdots,q)$ and  $D:=\cup_{j=1}^q D_j.$ Let $E$ be  the  $\mathbb{C}$-vector subspace of $H^{0}(X, L^{d})$ generated by       $\sigma_{1}^{\frac{d}{d_1}},\cdots,\sigma_{q}^{\frac{d}{d_q}}$. 
Put  $u=\rank E, \dim E=m+1.$ 
Denote by $R_j$ the zero divisors of $\sigma_j \ (j=1,2,\cdots,q).$  Assume that $R_1,\cdots,R_q$ are 
in $N$-subgeneral position in $X$, $B(E)=\varnothing$ and $q> (m+1)^2 K(E,N,\{D_j\}).$
Let $\mathcal{F} \subset Hol(\Omega,X)$ be given. Then $\mathcal{F}$ is an uniformly normal
family
if and only if the following two conditions hold

{\rm (i)} $\sup \Big\{ |df(p)|_{F_m} : p \in \bigcup\limits_{\mathcal{F}}
f^{-1}(D), f \in \mathcal{F} \Big\} < \infty$, and

{\rm (ii)} Each Brody limit $g$ for $\mathcal{F}$ such that 
$g(\mathbb{C})\subset \supp(\nu_{\sigma})$ for some $\sigma \in E\setminus \{0\},$ is constant. 
\end{theorem}

\begin{proof} $(\Rightarrow)$\ Suppose that $\mathcal{F}$ is uniformly normal.
The assertion is deduced from Corollary \ref{co710} and results of Joseph - Kwack
(see \cite[Theorem 3.4]{JK}).

$(\Leftarrow)$\  Now, assume that we have the conditions $(i)$ and $(ii).$ Suppose that $\mathcal{F}$ is not uniformly normal. By a result of Joseph - Kwack (see \cite[Theorem 3.4]{JK}), there exists
a nonconstant Brody limit $g \in Hol(\mathbb{C},X)$. This means that there
exist sequences $\{f_k\} \subset \mathcal{F}$ and 
$\{\varphi_k\} \subset Hol(\Delta_k,\Omega)$ such that the sequence
$\{g_k = f_k \circ \varphi_k\}$ converges uniformly to $g$.
By (ii), $g$ is analytically non-degenerate with respect to $E$. By the Ramification theorem and since $q> (m+1)^2 K(E,N,\{D_j\}),$ there exists $1\le j_0 \le q$ such that 
$$g^{*} R_j  \le m\cdot \supp  (g^{*} R_j).$$ 
Take $z_0 \in \supp  (g^{*} R_j).$ Then, since $g^{* }R_j  \le m\cdot \supp  (g^{*} R_j)$, $J_m(g)(z_0) \not =0.$ This implies that $F_m\circ J_k(g)(z_0)=\alpha>0.$
Thus, there exists an open set $U_0$ containing $z_0$, a neighbourhood  $V_0$ of $g(z_0)$ in $X$    such that  $g_k(z),g(z)\in V_0$ for all $z\in U_0,$ $\sigma_j |_{V_0}:=\sigma_{j0}$ is a holomorphic function on $V_0.$ 
By Hurwitz's Lemma,
there is a sequence $\{z_k\}$ in $\mathbb{C}$ such that $\{z_k\} \to z_0$,
$(\sigma_{j0} \circ g_k) (z_k) = 0$. We have
  $$ \lim_{k \to \infty} F_m\circ J_m(g_k)(z_k) =F_m \circ J_m(g)(z_0)=\alpha > 0.$$
Let $p_k = \varphi_k(z_k)$.
Then
$$ f_k(p_k) = (f_k \circ \varphi_k)(z_k) = g_k(z_k) \in
\sigma_{j0}^{-1}(0) \subset D_{j_0}, $$
and hence, $p_k \in \bigcup\limits_{\mathcal{F}} f^{-1}(D)$ for each $k \geq 1$.

On the other hand, by Definition \ref{de711} and Lemma \ref{le75}, we have
\begin{align*}
|df_k(p_k)| &\ge   \frac{f^* F_m \Big(p_k,j_m(\varphi_k))}{K^m_{\Omega}(p_k,j_m(\varphi_k))}= \frac{F_m\circ J_m(g_k)(z_k)}{K^m_{\Omega}(p_k,j_m(\varphi_k))}\\
& \ge \frac{F_m\circ J_m(g_k)(z_k)}{K^m_{\Delta_k}(z_k,j_m(id))}\ge k\cdot F_m\circ J_m(g_k)(z_k),\\
\end{align*}
so $|df_k(p_k)| \to \infty$. Since 
$\{p_k\} \subset \bigcup\limits_{\mathcal{F}} f^{-1}(D)$,
condition (i) does not hold.  This is a contradiction.
\end{proof}

\end{document}